\documentclass[preprint,12pt]{elsarticle}
\usepackage[margin=2.5cm]{geometry}

\usepackage{graphicx}
\usepackage[centertags]{amsmath}
\usepackage{amsfonts}
\usepackage{amssymb}
\usepackage{amsthm}
\usepackage{mathrsfs} %%% per la B corsiva di effortless
\usepackage{nicefrac} %%% diagonal fraction in the text
\usepackage{appendix}
\usepackage{epstopdf}
\usepackage{subfig}
\usepackage[table,xcdraw]{xcolor}

\usepackage{floatrow}
\newfloatcommand{capbtabbox}{table}[][\FBwidth]
\usepackage{pgf,tikz}
\usetikzlibrary{calc,patterns,angles,quotes,intersections,patterns}
\usepackage{tkz-euclide}
\usetkzobj{all}
\definecolor{green(ryb)}{rgb}{0.173, 0.627, 0.173}
\definecolor{azure}{rgb}{0.0, 0.5, 1.0}
\usepackage{verbatim}
\usepackage[norelsize,ruled,vlined,linesnumbered,titlenumbered]{algorithm2e}
\makeatletter
\newcommand{\nosemic}{\renewcommand{\@endalgocfline}{\relax}}% Drop semi-colon ;
\newcommand{\dosemic}{\renewcommand{\@endalgocfline}{\algocf@endline}}% Reinstate semi-colon
\let\oldnl\nl
\newcommand{\nonl}{\renewcommand{\nl}{\let\nl\oldnl}}
\makeatother
\newcommand{\highlight}[1]{\textbf{#1}}
\newtheorem{thm}{Theorem}[section]
\newtheorem{prop}[thm]{Proposition}
\newtheorem{definition}[thm]{Definition}

\newtheorem{cor}[thm]{Corollary}
\newtheorem{rec}[thm]{Recipe}
\def\RR{{\mathbb{R}}}

\newcommand{\cM}{{\mathcal{M}}}
\newcommand{\cB}{{\mathcal{B}}}

\newcommand{\cE}{{\mathcal{E}}}

\newcommand{\supp}{\text{supp}\,}

\newcommand{\NtwoS}{{N$_2$S}}
\newcommand{\NtwoStwo}{{N$_2$S$_2$}}
\newcommand{\cN}{\mathcal{N}}
\newcommand{\unilateral}{one-directional}

\journal{TBA}

\begin{document}

\begin{frontmatter}


%\title{Locally linearly independent LR B-splines and applications to quasi-interpolation and isogeometric analysis\tnoteref{funding}}
%\title{Locally linearly independent LR B-splines: Adaptive refinement and applications\tnoteref{funding}}
\title{Adaptive refinement with locally linearly independent LR B-splines: Theory and applications\tnoteref{funding}}
\tnotetext[funding]{This work has received funding from the European Union's Horizon 2020 Research and Innovation Programme (grant agreement No 675789)
and from the MIUR Excellence Department Project awarded to the Department of Mathematics, University of Rome Tor Vergata (CUP E83C18000100006).
The authors are members of Gruppo Nazionale per il Calcolo Scientifico, Istituto Nazionale di Alta Matematica.}

\author[Si,uio]{Francesco Patrizi\corref{cor1}}
\ead{francesco.patrizi@sintef.no}
\author[TV]{Carla Manni}
\ead{manni@mat.uniroma2.it}
\author[TV]{Francesca Pelosi}
\ead{pelosi@mat.uniroma2.it}
\author[TV]{Hendrik Speleers}
\ead{speleers@mat.uniroma2.it}
\address[Si]{Department of Applied Mathematics and Cybernetics, SINTEF, Oslo, Norway}
\address[uio]{Department of Mathematics, University of Oslo, Oslo, Norway}
\address[TV]{Department of Mathematics, University of Rome Tor Vergata, Rome, Italy}
\cortext[cor1]{Corresponding Author}

\begin{abstract}
In this paper we describe an adaptive refinement strategy for LR B-splines. The presented strategy ensures, at each iteration, local linear independence of the obtained set of LR B-splines. This property is then exploited in two applications: the construction of efficient quasi-interpolation schemes and the numerical solution of elliptic problems using the isogeometric Galerkin method.
\end{abstract}

\end{frontmatter}
%\linenumbers


%------------------------------------------------------------------------

\section{Introduction}
Since the '70s, curves and surfaces in engineering are usually expressed by means of computer aided design (CAD) technologies, such as B-splines and non-uniform rational B-splines (NURBS). Thanks to properties like nonnegativity, local support and partition of unity, they allow for an easy control and modification of the geometries they describe,
and this motivates their undisputed success as main modeling tools for objects with complex shapes in engineering; see, e.g., \cite{PieglT1997,CohenRE2001,cetraro} and references therein.
On the other hand, B-splines also provide a very efficient representation of smooth piecewise polynomial spaces, and so are a popular ingredient in the construction of  approximation schemes; see, e.g., \cite{deboor,schumaker,lychemanni} and references therein.

More recently, the advent of isogeometric analysis (IgA) has integrated spline and CAD technologies into finite element analysis (FEA); see, e.g.,  \cite{IgABook,IgAacta}. IgA aims to unify the geometric description of the domain of the differential problem with its numerical resolution, in order to expedite the simulation process and gaining in accuracy. In addition to the properties listed above, B-splines and NURBS feature other qualities appreciated in this context, such as (local) linear independence and high global smoothness.

Tensor structures admit a simple but powerful multivariate extension of univariate splines and B-splines. On the other hand, they lack adequate local refinement. The constantly increasing demand for higher precision in simulations and reverse engineering processes requires the possibility to apply adaptive local refinement strategies, in order to reduce the approximation error while keeping the computational costs low.
This request for adaptivity, triggered the interest in new formulations of B-splines and NURBS, still based on local tensor structures \cite{HBsplines,Tsplines,PHTsplines,ASTsplines,THBsplines,paperLR,PBsplines}. All these new classes of functions are defined on locally refined meshes, in which T-vertices in the interior of the domain are allowed, the so-called T-meshes.

Locally refined B-splines, or in short LR B-splines \cite{paperLR}, are one of these new formulations, and their definition is inspired by the knot insertion refinement process of tensor B-splines. These latter are defined on global knot vectors, one per direction. The insertion of a new knot in a knot vector corresponds to a line segment in the mesh crossing the entire domain. This refines all the B-splines whose supports are crossed. Instead, LR B-splines are defined on local knot vectors and the insertion of a new knot is always performed with respect to a particular LR B-spline.
As a consequence, the LR B-spline definition is consistent with the tensor B-spline definition when the underlying mesh at the end of the process is a tensor mesh, and the formulation of LR B-splines remains broadly similar to classical tensor B-splines even though they allow for local refinements. This makes them one of the most elegant extensions of univariate B-splines on local tensor structures.

LR B-splines possess almost all the properties of classical tensor B-splines. Unfortunately, they are not always linearly independent.
To this day, it is not yet known what are the precise conditions on the locally refined mesh to ensure a linearly independent set of LR B-splines, but some progress has been made in this direction.
In \cite{paperLR} an efficient algorithm to seek and destroy linear dependence relations has been introduced, but it does not handle every possible locally refined mesh. In \cite{LRdependence} a first analysis on the necessary conditions for encountering a linear dependence relation has been presented. There, it has also been proved that, for any bidegree, a linear dependence relation in the LR B-spline set involves at least 8 functions.
In \cite{someproperties,meshbressan} a characterization of the local linear independence of LR B-splines has been provided. Such a strong property is guaranteed only on locally refined meshes with strong constraints on the lengths and positions of the line segments that yield particular arrangements of the LR B-spline supports.
On the other hand, a practical adaptive refinement strategy to produce meshes with the local linear independence property is still missing in the literature. To the best of our knowledge, the only mesh construction that leads to a locally linearly independent set of LR B-splines can be found in \cite{meshbressan}. Such a construction, however, cannot be considered as a practical strategy because the regions to be refined and the maximal resolution, i.e., the sizes of the smallest cells in the domain induced by the mesh, must be chosen a priori.

In this paper, we describe a practical refinement strategy ensuring the local linear independence of the corresponding LR B-splines. Such a property is appealing as it admits, e.g., the construction of efficient quasi-interpolation schemes and the unisolvency of linear systems obtained by isogeometric discretization of differential problems based on such LR-splines.
The remainder of the paper is divided into 5 sections. Section~\ref{sec:LR-splines} contains the definition of LR B-splines and a summary of their main properties, whereas Section~\ref{sec:refinement} describes the mesh refinement strategy and is the core of the paper. Sections~\ref{sec:quasi-interpolation} and~\ref{sec:isogeometric-analysis} present applications of the refinement strategy in the context of quasi-interpolation and  isogeometric Galerkin discretizations of elliptic problems. We end in Section~\ref{sec:conclusion} with some concluding remarks.

Throughout the paper, we assume the reader to be familiar with the definition and main properties of (univariate) B-splines, in particular with the knot insertion procedure. An introduction to this topic can be found, e.g., in the review papers \cite{cetraro,lychemanni} or in the classical books \cite{deboor} and \cite{schumaker}.

%------------------------------------------------------------------------

\section{Locally refined B-splines}\label{sec:LR-splines}
In this section, we introduce locally refined B-splines, or in short LR B-splines, and discuss several of their properties, following the terminology
 from \cite{LRdependence}.
We denote by $\Pi_p$ the space of univariate polynomials of degree less than or equal to $p$, and by
 $\Pi_{\pmb{p}}$ the space of bivariate polynomials of degrees less than or equal to $\pmb{p} = (p_1, p_2)$ component-wise. Furthermore, we denote by
$B[\pmb{x}, \pmb{y}]$ the bivariate B-spline defined on the (local) knot vectors $\pmb{x} = (x_1, \ldots, x_{p_1+2})$ and $\pmb{y}=(y_1, \ldots, y_{p_2+2})$, where $x_i\leq x_{i+1}$ and $y_i\leq y_{i+1}$ for all $i$. The bidegree of $B[\pmb{x}, \pmb{y}]$ is $\pmb{p}=(p_1,p_2)$
and is implicitly specified by the length of $\pmb{x}$ and $\pmb{y}$.

In order to define LR B-splines, we first introduce the concept of box-partition.
\begin{definition}
Given an axis-aligned rectangle $\Omega \subseteq \RR^2$, a \highlight{box-partition} of $\Omega$ is a finite collection $\mathcal{E}$ of axis-aligned rectangles in $\Omega$ such that
\begin{enumerate}
\item $\mathring{\beta}_1 \cap \mathring{\beta}_2 = \emptyset$ for any $\beta_1, \beta_2 \in \mathcal{E}$, with $\beta_1 \neq \beta_2$;
\item $\bigcup_{\beta \in \mathcal{E}} \beta = \Omega$.
\end{enumerate}
\end{definition}
Given a box-partition $\mathcal{E}$, we define the \highlight{vertices} of $\mathcal{E}$ as the vertices of its elements.
%In particular, a vertex of $\mathcal{E}$ is called \highlight{T-vertex} if it is the intersection of three element edges.
A \highlight{meshline} is an axis-aligned segment contained in an edge of an element of $\mathcal{E}$, connecting two and only two vertices of $\mathcal{E}$ located at its end-points. The collection of all the meshlines of the box-partition is called \highlight{mesh}, and denoted by $\cM$.
A meshline can be expressed as the Cartesian product of a point in $\RR$ and a finite interval. Let $\alpha \in \RR$ be the value of such a point and let $k \in \{1,2\}$ be its position in the Cartesian product. If $k=1$ the meshline is vertical and if $k=2$ the meshline is horizontal. We sometimes write $k$-meshline to specify the direction of the meshline, and $(k,\alpha)$-meshline to specify exactly on which axis-parallel line in $\RR^2$ the meshline lies.
A vertex of $\mathcal{E}$ is called \highlight{T-vertex} if it is the intersection of two collinear meshlines and another meshline, say $\gamma$, orthogonal to them.
We call the T-vertex {vertical} if $\gamma$ is vertical,
and {horizontal} otherwise.

For defining splines of a certain bidegree $\pmb{p}=(p_1,p_2)$ and smoothness across the meshlines, we also need the notion of \highlight{multiplicity} of a meshline.
This is a positive integer associated with every meshline in $\cM$. For a $k$-meshline this number is assumed to be maximally $p_k+1$.
A meshline in $\cM$ has \highlight{full multiplicity} if its multiplicity is maximal, and we say that $\cM$ is \highlight{open} if every boundary meshline has full multiplicity.
If all the meshlines of the box-partition have the same multiplicity $m$ we say that $\cM$ has multiplicity $m$.
When the T-vertices of $\mathcal{E}$ occur only on $\partial \Omega$ and all collinear meshlines have the same multiplicity, the corresponding mesh is called \highlight{tensor mesh}.
Figure~\ref{boxp} shows an example of a box-partition and its associated mesh.

\begin{figure}[t!]\centering
\subfloat[]{
\begin{tikzpicture}[scale=3.1]
\fill[red!60] (0,.66) -- (.5,.66) -- (.5, 1) -- (0,1) -- cycle;
\fill[cyan!50] (0,0) -- (.5,0) -- (.5, .66) -- (0,.66) -- cycle;
\fill[orange!60] (.5, 0) -- (1,0) -- (1,.33) -- (.5, .33) --cycle;
\fill[yellow!60] (.5, .33) -- (1,.33) -- (1,.66) -- (.5, .66) --cycle;
\fill[green!60] (.5, .66) -- (1.5,.66) -- (1.5,1) -- (.5, 1) --cycle;
\fill[purple!60] (1, 0) -- (1.5,0) -- (1.5,.66) -- (1, .66) --cycle;
\draw (0,0) -- (1.5,0) -- (1.5, 1) -- (0,1) --cycle;
\draw (0,.66) --(1.5, .66);
\draw (.5, .33) -- (1, .33);
\draw (.5, 0) -- (.5, 1);
\draw (1, 0) -- (1, .66);
\draw (.25, .83) node{$\beta_1$};
\draw (1, .83) node{$\beta_2$};
\draw (.25, .33) node{$\beta_3$};
\draw (.75, .17) node{$\beta_5$};
\draw (.75, .5) node{$\beta_4$};
\draw (1.25, .33) node{$\beta_6$};
\draw[white] (0,-.08);
\end{tikzpicture}}\qquad\qquad
\subfloat[]{
\begin{tikzpicture}[scale=3.1]
\tikzset{%
  block/.style    = {draw=black, fill=white, rectangle, minimum height = .05cm,
    minimum width = .05cm}
    }
\draw (0,0) -- (1.5,0) -- (1.5, 1) -- (0,1) --cycle;
\draw (0,.66) --(1.5, .66);
\draw (.5, .33) -- (1, .33);
\draw (.5, 0) -- (.5, 1);
\draw (1, 0) -- (1, .66);
\draw (0,.33) node[block]{\footnotesize{$1$}};
\draw (0,.845) node[block]{\footnotesize{$1$}};
\draw (.25,0) node[block]{\footnotesize{$1$}};
\draw (.25,.66) node[block]{\footnotesize{$1$}};
\draw (.25,1) node[block]{\footnotesize{$1$}};
\draw (.5,.165) node[block]{\footnotesize{$1$}};
\draw (.5, .495) node[block]{\footnotesize{$3$}};
\draw (.5,.845) node[block]{\footnotesize{$1$}};
\draw (.75,.66) node[block]{\footnotesize{$2$}};
\draw (.75,.33) node[block]{\footnotesize{$4$}};
\draw (.75,0) node[block]{\footnotesize{$2$}};
\draw (1.25,.66) node[block]{\footnotesize{$2$}};
\draw (1,1) node[block]{\footnotesize{$1$}};
\draw (1,.165) node[block]{\footnotesize{$1$}};
\draw (1,.495) node[block]{\footnotesize{$1$}};
\draw (1.5,.33) node[block]{\footnotesize{$1$}};
\draw (1.5,.845) node[block]{\footnotesize{$1$}};
\draw (1.25,0) node[block]{\footnotesize{$2$}};
\end{tikzpicture}
}
\caption{Example of a box-partition $\cE$ of a rectangle $\Omega$ in (a), and the mesh corresponding to $\cE$ in (b). The meshlines are identified by squares showing the associated multiplicities.}\label{boxp}
\end{figure}
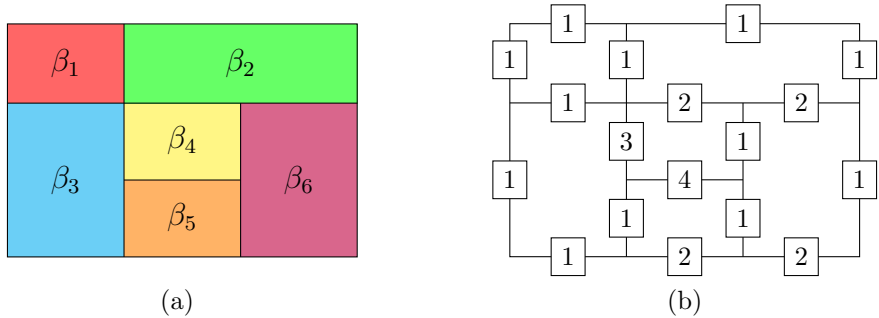

%Given a mesh $\cM$, we now define the B-splines whose support are on $\cM$.
%A bivariate tensor B-spline defined by the knot vectors $\pmb{x}$ and $\pmb{y}$ is denoted by $B[\pmb{x}, \pmb{y}]$. Let
Given a bivariate B-spline $B[\pmb{x}, \pmb{y}]$, let $x_{i_1}, \ldots, x_{i_r}$ and $y_{j_1}, \ldots, y_{j_s}$ be the distinct knots in $\pmb{x}$ and $\pmb{y}$, respectively. The mesh \begin{equation}\label{M(x,y)}
\begin{split}
\cM(\pmb{x}, \pmb{y}) := \{&\{x_{i_\ell}\} \times [y_{j_n}, y_{j_{n+1}}] : \,\ell=1, \ldots, r;\, n=1, \ldots, s-1\}\,\\ &\cup\,\{[x_{i_n}, x_{i_{n+1}}]\times\{y_{j_\ell}\} : \ell =1,\ldots, s;\, n=1, \ldots, r-1\}
\end{split}
\end{equation}
is a tensor mesh in $\supp B[\pmb{x}, \pmb{y}]$. The multiplicities of the meshlines in $\cM(\pmb{x}, \pmb{y})$ are given by the multiplicities of the knots of $B[\pmb{x}, \pmb{y}]$. For instance, the $(1,x_{i_\ell})$-meshlines $\{x_{i_\ell}\}\times[y_{j_n}, y_{j_{n+1}}]$ for $n=1, \ldots, s-1$ have all the same multiplicity equal to the multiplicity of $x_{i_\ell}$ in $\pmb{x}$.
\begin{definition}
Given a mesh $\cM$ and a B-spline $B[\pmb{x}, \pmb{y}]$, we say that $B[\pmb{x}, \pmb{y}]$ has \highlight{support} on $\cM$ if
\begin{itemize}
\item the meshlines in $\cM(\pmb{x}, \pmb{y})$ can be obtained as unions of meshlines in $\cM$,
and
\item their multiplicities are less than or equal to the multiplicities of the corresponding meshlines in $\cM$.
\end{itemize}
Furthermore, we say that $B[\pmb{x}, \pmb{y}]$ has \highlight{minimal support} on $\cM$ if
\begin{itemize}
\item it has support on $\cM$,
\item the multiplicities of the interior meshlines in $\cM(\pmb{x}, \pmb{y})$ are equal to the multiplicities of the corresponding meshlines in $\cM$, and
\item there is no collection $\gamma$ of collinear meshlines in $\cM\backslash\cM(\pmb{x}, \pmb{y})$ such that $\supp B[\pmb{x}, \pmb{y}]\backslash \gamma$ is not connected.
\end{itemize}
\end{definition}
Figure~\ref{exMS} shows examples of B-splines of bidegree $(2,2)$ with support on a mesh of multiplicity $1$. In particular, the B-splines in (b)--(c) have minimal support, whereas the support of the B-spline in (d) can be disconnected by the collection of meshlines $\gamma$, visualized by dashed lines in the figure.
\begin{figure}[t!]
    \centering
\subfloat[]{\centering
\begin{tikzpicture}[scale=3.1]
\draw (0,0) -- (1,0) -- (1,1) -- (0,1) -- cycle;
\draw (0,.2) --(1, .2);
\draw (0,.8) -- (1,.8);
\draw (.6, 0) -- (.6, 1);
\draw (.8, 0) -- (.8, 1);
\draw (0,.6) -- (.8, .6);
\draw (.2, .2) -- (.2, 1);
\draw (0, .7) -- (.6, .7);
\draw (.4, 0) -- (.4, .8);
\draw (.5, .2) -- (.5, 1);
\end{tikzpicture}}\qquad
\subfloat[]{\centering
\begin{tikzpicture}[scale=3.1]
\filldraw[draw = red,ultra thick, fill =red!60] (0,0) -- (.8,0) -- (.8,.8) -- (0,.8) -- cycle;
\draw[red,ultra thick] (0, .2) -- (.8, .2);
\draw[red,ultra thick] (0,.6) -- (.8, .6);
\draw[red,ultra thick] (.4,0) -- (.4, .8);
\draw[red,ultra thick] (.6, 0) -- (.6, .8);
\draw (0,0) -- (1,0) -- (1,1) -- (0,1) -- cycle;
\draw (0,.2) --(1, .2);
\draw (0,.8) -- (1,.8);
\draw (.6, 0) -- (.6, 1);
\draw (.8, 0) -- (.8, 1);
\draw (0,.6) -- (.8, .6);
\draw (.2, .2) -- (.2, 1);
\draw (0, .7) -- (.6, .7);
\draw (.4, 0) -- (.4, .8);
\draw (.5, .2) -- (.5, 1);
\end{tikzpicture}}\qquad
\subfloat[]{\centering
\begin{tikzpicture}[scale=3.1]
\filldraw[draw = red,ultra thick, fill =red!60] (0,0) -- (1,0) -- (1,1) -- (0,1) -- cycle;
\draw[red,ultra thick] (0, .2) -- (1, .2);
\draw[red,ultra thick] (0,.8) -- (1, .8);
\draw[red,ultra thick] (.6,0) -- (.6, 1);
\draw[red,ultra thick] (.8, 0) -- (.8, 1);
\draw (0,0) -- (1,0) -- (1,1) -- (0,1) -- cycle;
\draw (0,.2) --(1, .2);
\draw (0,.8) -- (1,.8);
\draw (.6, 0) -- (.6, 1);
\draw (.8, 0) -- (.8, 1);
\draw (0,.6) -- (.8, .6);
\draw (.2, .2) -- (.2, 1);
\draw (0, .7) -- (.6, .7);
\draw (.4, 0) -- (.4, .8);
\draw (.5, .2) -- (.5, 1);
\end{tikzpicture}}\qquad
\subfloat[]{\centering
\begin{tikzpicture}[scale=3.1]
\filldraw[draw = red,ultra thick, fill =red!60] (0,.2) -- (.6,.2) -- (.6,1) -- (0,1) -- cycle;
\draw[red,ultra thick] (0, .6) -- (.6, .6);
\draw[red,ultra thick] (0,.8) -- (.6, .8);
\draw[red,ultra thick] (.2,.2) -- (.2, 1);
\draw[red,ultra thick] (.5, .2) -- (.5, 1);
\draw (0,0) -- (1,0) -- (1,1) -- (0,1) -- cycle;
\draw (0,.2) --(1, .2);
\draw (0,.8) -- (1,.8);
\draw (.6, 0) -- (.6, 1);
\draw (.8, 0) -- (.8, 1);
\draw (0,.6) -- (.8, .6);
\draw (.2, .2) -- (.2, 1);
\draw (0, .7) -- (.6, .7);
\draw (.4, 0) -- (.4, .8);
\draw (.5, .2) -- (.5, 1);
\draw[blue, dashed, ultra thick] (0, .7) -- (.6, .7);
\end{tikzpicture}}
\caption{Support of B-splines of bidegree $(2,2)$ on a mesh $\cM$ of multiplicity $1$. The mesh is shown in (a). The B-splines whose supports are depicted in (b) and (c) have minimal support on $\cM$. The tensor meshes defined by the B-spline's knots are highlighted with thicker lines. On the other hand, the B-spline in (d) does not have minimal support on $\cM$: the collection of meshlines contained in the dashed line disconnects its support.}\label{exMS}
\end{figure}
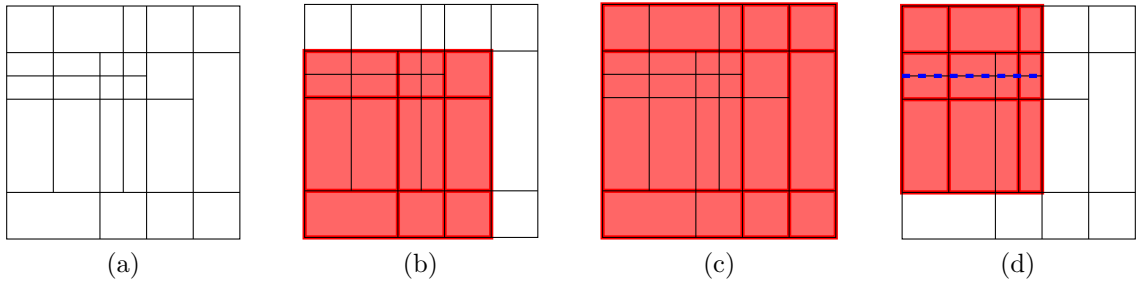

Given a mesh $\cM$ and a B-spline $B[\pmb{x}, \pmb{y}]$ with support in $\cM$, assume that it does not have minimal support on $\cM$. Then, there exists a collection of $(k,\alpha)$-meshlines $\gamma$ such that $\supp B[\pmb{x}, \pmb{y}] \backslash \gamma$ is not connected and either $\gamma$ is in $\cM \backslash \cM(\pmb{x}, \pmb{y})$ or $\gamma \subseteq \cM(\pmb{x}, \pmb{y})$, i.e., $\alpha$ is an internal knot of $\pmb{x}$ for $k=1$ or $\pmb{y}$ for $k=2$, but its meshlines have lower multiplicities in $\cM(\pmb{x}, \pmb{y})$ than in $\cM$. Assume that the meshlines in $\gamma$ have all the same multiplicity $m$ in $\cM$.
Denoting by $\mu(\alpha) \geq 0$ the number of times $\alpha$ appears in the knot vector of $B[\pmb{x}, \pmb{y}]$ in the $k$-th direction, then $m - \mu(\alpha)$ is strictly positive as $B[\pmb{x}, \pmb{y}]$ has support, but not minimal support, on $\cM$. One could consider such $\alpha$ as an extra knot, of multiplicity $m-\mu(\alpha)$, with respect to the knot vector of $B[\pmb{x},\pmb{y}]$ in the $k$-th direction (in $\pmb{x}$ if $k=1$ and in $\pmb{y}$ if $k=2$), and perform knot insertion on $B[\pmb{x},\pmb{y}]$. If $\alpha$ was already a knot of $B[\pmb{x}, \pmb{y}]$, so $\mu(\alpha) \geq 1$, this means rising its multiplicity by $m-\mu(\alpha)$. The resulting generated B-splines will still have support on $\cM$ and eventually they will also have minimal support on $\cM$. As an example, the collection $\gamma$ highlighted with dashed lines in Figure~\ref{exMS}(d) is made of $(2,\alpha)$-meshlines, for some $\alpha$, of multiplicity 1. Such $\alpha$ can be inserted as new knot of multiplicity 1 in the knot vector in the $y$-direction of the considered B-spline to refine it in two B-splines via knot insertion.

The LR B-splines are generated by means of the above procedure. We start by considering a coarse tensor mesh and we refine it by inserting collections of collinear meshlines, one at a time, of the same multiplicity. On the initial mesh we consider the standard tensor B-splines and whenever a B-spline in our collection has no longer minimal support during the mesh refinement process, we refine it by using the knot insertion procedure. The LR B-splines will be the final set of B-splines produced by this algorithm.
In the following definitions we formalize this by describing the mesh refinement process in our framework.
\begin{definition}
Given a box-partition $\mathcal{E}$ and an axis-aligned segment $\gamma$, we say that $\gamma$ \highlight{traverses} $\beta \in \mathcal{E}$ if $\gamma \subseteq \beta$ and the interior of $\beta$ is divided into two parts by $\gamma$, i.e., $\beta\backslash \gamma$ is not connected.
A \highlight{split} is a finite union of contiguous and collinear axis-aligned segments $\gamma = \cup_i \gamma_i$ such that every $\gamma_i$ either is a meshline of the box-partition or traverses some $\beta\in \mathcal{E}$.  A mesh $\cM$ has \highlight{constant splits} if each split in it is made of meshlines of the same multiplicity.
\end{definition}
Like for meshlines, we sometimes write $k$-split with $k \in \{1,2\}$ to specify the direction of the split or $(k,\alpha)$-split to specify on what axis-parallel line in $\RR^2$ the split lies.

When a split $\gamma$ is inserted in a box-partition $\mathcal{E}$, any traversed $\beta \in \mathcal{E}$ is replaced by the two subrectangles $\beta_1, \beta_2$ given by the closures of the connected components of $\beta \backslash \gamma$. The resulting new box-partition will be denoted by $\mathcal{E}+\gamma$ and its corresponding mesh by $\cM+\gamma$.
We also assume that a positive integer $\mu_\gamma$ has been assigned to any split $\gamma$. The multiplicities of the meshlines in $\cM\cap(\cM+\gamma)$ and not contained in $\gamma$ are unchanged. Contrarily,
the multiplicities of the meshlines in $\gamma$ that were already in $\cM$ are rised by $\mu_\gamma$, and the new meshlines in $\gamma$ have multiplicity equal to $\mu_\gamma$ on $\cM + \gamma$.

The LR B-splines are defined on a class of meshes with constant splits, called {LR-meshes}. Thus, from now on, we restrict our attention to meshes that have constant splits. In particular, we note that when refining a mesh $\cM$ by inserting a split $\gamma$, either $\gamma$ is made solely of new meshlines or it is made solely of meshlines already on $\cM$, in order for $\cM + \gamma$ to have constant splits.

\begin{definition}
Given a mesh $\cM$ with constant splits, a B-spline $B[\pmb{x}, \pmb{y}]$ with support on $\cM$ and a split $\gamma$, we say that $\gamma$ \highlight{traverses} $B[\pmb{x}, \pmb{y}]$ if the interior of $\supp B[\pmb{x}, \pmb{y}]$ is divided into two parts by $\gamma$, i.e., $\supp B[\pmb{x}, \pmb{y}]\backslash \gamma$ is not connected and either $\gamma$ is in $\cM \backslash \cM(\pmb{x}, \pmb{y})$ or $\gamma \subseteq \cM(\pmb{x}, \pmb{y})$ but the multiplicity of its meshlines is lower in $\cM(\pmb{x}, \pmb{y})$ than in $\cM$.%%% actually we have not defined the multiplicity of a split but of its meshlines, so we should say ``but it is mad of meshlines that have different....''
\end{definition}

We are now ready to define the mesh refinement process and the LR B-splines. The meshes generated by this procedure will be called LR-meshes.
\begin{definition}\label{defLR}
Given a bidegree $\pmb{p}=(p_1,p_2)$, let $\cM_1$ be a tensor mesh such that the set of standard tensor B-splines of bidegree $\pmb{p}$ on $\cM_1$ is nonempty, and denote it by $\mathcal{B}_1$. We then define a sequence of meshes $\cM_{2},\cM_{3},\ldots$ and corresponding function sets $\mathcal{B}_2,\mathcal{B}_3,\ldots$ as follows.
For $i=1,2,\ldots$, let $\gamma_i$ be a split such that $\cM_{i+1} := \cM_i + \gamma_i$ has constant splits and such that the support of at least one B-spline in $\cB_i$ is traversed by a split in $\cM_{i+1}$. On this refined mesh $\cM_{i+1}$, the new set of B-splines $\cB_{i+1}$ is constructed by the following algorithm.
\begin{enumerate}
  \item Initialize the set by $\mathcal{B}_{i+1} \leftarrow \mathcal{B}_i$.
  \item As long as there exists $B[\pmb{x}^j, \pmb{y}^j]\in \mathcal{B}_{i+1}$ with no minimal support on $\cM_{i+1}$:
  \begin{enumerate}
    \item Apply knot insertion: $\exists B[\pmb{x}_{1}^j, \pmb{y}_{1}^j], B[\pmb{x}_{2}^j, \pmb{y}_{2}^j] : B[\pmb{x}^j, \pmb{y}^j] = \alpha_1B[\pmb{x}_1^{j}, \pmb{y}_1^{j}]+ \alpha_2 B[\pmb{x}_2^{j}, \pmb{y}_2^{j}]$.\vspace{.15cm}
    \item Update the set: $\mathcal{B}_{i+1} \leftarrow (\mathcal{B}_{i+1} \backslash \{B[\pmb{x}^j, \pmb{y}^j]\})\cup\{B[\pmb{x}_1^{j}, \pmb{y}_1^{j}], B[\pmb{x}_2^{j}, \pmb{y}_2^{j}]\}$.
  \end{enumerate}
\end{enumerate}
The mesh generated at each step is called \highlight{LR-mesh} and the corresponding function set is called \highlight{LR B-spline set}.
\end{definition}

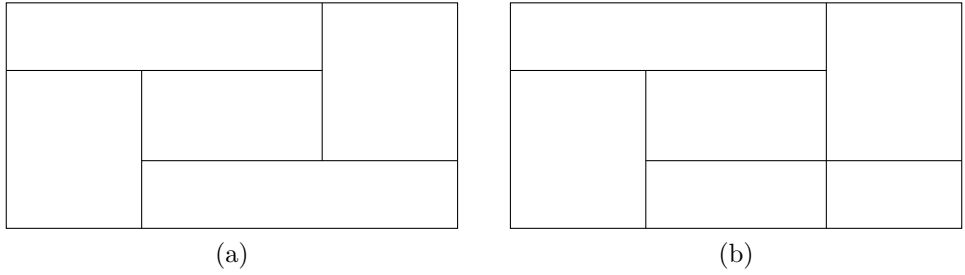
\begin{figure}[t!]
\centering
\subfloat[]{
\begin{tikzpicture}[scale=6]
\draw (0,0) -- (1,0) -- (1,.5) -- (0,.5) -- cycle;
\draw (0,.35) -- (.7,.35);
\draw (.7,.5) -- (.7,.15);
\draw (1,.15) -- (.3,.15);
\draw (.3,0) -- (.3,.35);
\end{tikzpicture}
}\quad
\subfloat[]{
\begin{tikzpicture}[scale=6]
\draw (0,0) -- (1,0) -- (1,.5) -- (0,.5) -- cycle;
\draw (0,.35) -- (.7,.35);
\draw (.7,.5) -- (.7,0);
\draw (1,.15) -- (.3,.15);
\draw (.3,0) -- (.3,.35);
\end{tikzpicture}
}
\caption{Two meshes. Assume that the boundary has a multiplicity large enough so that it is possible to define a B-spline of bidegree $\pmb{p}$ on it. Then, the mesh in (a) is not an LR-mesh because it cannot be built as a sequence of split insertions. The mesh in (b) is an LR-mesh similar to the one in (a).}\label{notLRmesh}
\end{figure}

Obviously not every mesh is an LR-mesh. For instance, one could consider meshes that do not have constant splits or meshes that cannot be built through a sequence of split insertions as the mesh depicted in Figure~\ref{notLRmesh}(a).
In general, the mesh refinement process producing a given LR-mesh $\cM=\cM_N$ is not unique. Indeed, the split insertion ordering can often be changed. Nevertheless, the LR B-spline set on $\cM$ is well defined because it is independent of such insertion ordering, as proved in \cite[Theorem~3.4]{paperLR}.

Given an LR-mesh, the corresponding LR B-splines have several desirable properties for applications. By their definition, it is clear that
\begin{itemize}
  \item they are nonnegative,
  \item they have minimal support, and
  \item they can be expressed by the LR B-splines on finer LR-meshes using nonnegative coefficients (provided by the knot insertion procedure).
\end{itemize}
Furthermore, it is possible to scale them by means of positive weights so that they also form a partition of unity; see \cite[Section 7]{paperLR}.

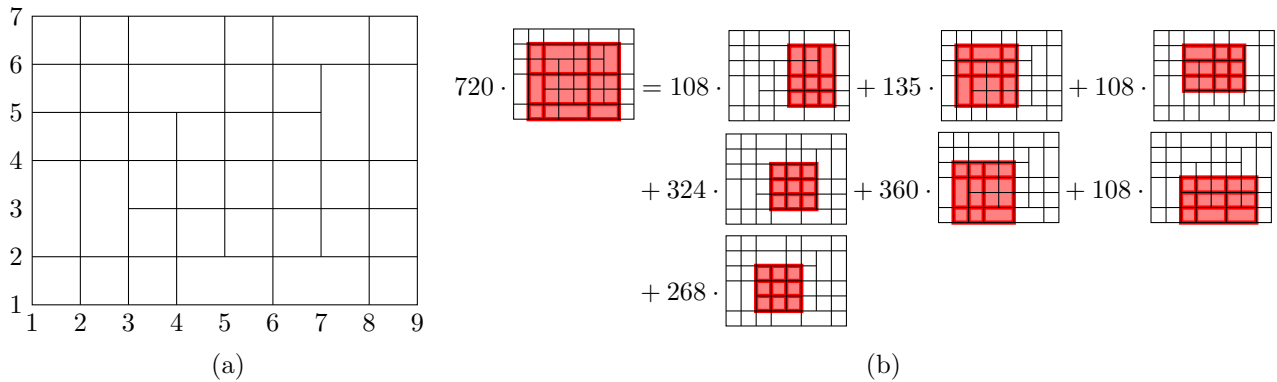
\begin{figure}[t!]
\subfloat[]{\hspace{-0.3cm}
\begin{tikzpicture}[scale=3.2]\footnotesize
\draw (0,0) -- (1.6,0) -- (1.6,1.2) -- (0,1.2) -- cycle;
\draw (.2,0) -- (.2,1.2);
\draw (.4,0) -- (.4, 1.2);
\draw (.6,0) -- (.6,.8);
\draw (.8,.2) -- (.8,1.2);
\draw (1,0) -- (1,1.2);
\draw (1.2,.2) -- (1.2,1);
\draw (1.4,0) -- (1.4,1.2);
\draw (0,.2) -- (1.6,.2);
\draw (.4,.4) -- (1.6,.4);
\draw (0,.6) -- (1.6,.6);
\draw (0,.8) -- (1.2,.8);
\draw (0,1) -- (1.6,1);
\draw (0,0) node[below]{1};
\draw (.2,0) node[below]{2};
\draw (.4,0) node[below]{3};
\draw (.6,0) node[below]{4};
\draw (.8,0) node[below]{5};
\draw (1,0) node[below]{6};
\draw (1.2,0) node[below]{7};
\draw (1.4,0) node[below]{8};
\draw (1.6,0) node[below]{9};
\draw (0,0) node[left]{1};
\draw (0,.2) node[left]{2};
\draw (0,.4) node[left]{3};
\draw (0,.6) node[left]{4};
\draw (0,.8) node[left]{5};
\draw (0,1) node[left]{6};
\draw (0,1.2) node[left]{7};
\end{tikzpicture}
}
\subfloat[]{\hspace{-3.5cm}
\begin{minipage}{0.9\textwidth}\vspace{-4.5cm}\footnotesize
\begin{equation*}\begin{split}
    \hspace{3cm}\footnotesize{720}\cdot\raisebox{-1em}{\begin{tikzpicture}
\filldraw[draw=red, ultra thick, fill=red!50] (.2,0) -- (1.4, 0) -- (1.4, 1) -- (.2,1) -- cycle;
\draw[red, ultra thick] (.2,.2) -- (1.4,.2);
\draw[red, ultra thick] (.2,.6) -- (1.4,.6);
\draw[red, ultra thick] (.4,0) -- (.4,1);
\draw[red, ultra thick] (1,0) -- (1,1);
\draw (0,0) -- (1.6,0) -- (1.6,1.2) -- (0,1.2) -- cycle;
\draw (.2,0) -- (.2,1.2);
\draw (.4,0) -- (.4, 1.2);
\draw (.6,0) -- (.6,.8);
\draw (.8,.2) -- (.8,1.2);
\draw (1,0) -- (1,1.2);
\draw (1.2,.2) -- (1.2,1);
\draw (1.4,0) -- (1.4,1.2);
\draw (0,.2) -- (1.6,.2);
\draw (.4,.4) -- (1.6,.4);
\draw (0,.6) -- (1.6,.6);
\draw (0,.8) -- (1.2,.8);
\draw (0,1) -- (1.6,1);
\end{tikzpicture}} &= \footnotesize{108}\cdot\raisebox{-1em}{\begin{tikzpicture}
\filldraw[draw=red, ultra thick, fill=red!50] (.8,.2) -- (1.4, .2) -- (1.4, 1) -- (.8,1) -- cycle;
\draw[red, ultra thick] (.8,.4) -- (1.4,.4);
\draw[red, ultra thick] (.8,.6) -- (1.4,.6);
\draw[red, ultra thick] (1,.2) -- (1,1);
\draw[red, ultra thick] (1.2,.2) -- (1.2,1);
\draw (0,0) -- (1.6,0) -- (1.6,1.2) -- (0,1.2) -- cycle;
\draw (.2,0) -- (.2,1.2);
\draw (.4,0) -- (.4, 1.2);
\draw (.6,0) -- (.6,.8);
\draw (.8,.2) -- (.8,1.2);
\draw (1,0) -- (1,1.2);
\draw (1.2,.2) -- (1.2,1);
\draw (1.4,0) -- (1.4,1.2);
\draw (0,.2) -- (1.6,.2);
\draw (.4,.4) -- (1.6,.4);
\draw (0,.6) -- (1.6,.6);
\draw (0,.8) -- (1.2,.8);
\draw (0,1) -- (1.6,1);
\end{tikzpicture}}+\footnotesize{135}\cdot \raisebox{-1em}{\begin{tikzpicture}
\filldraw[draw=red, ultra thick, fill=red!50] (.2,.2) -- (1, .2) -- (1, 1) -- (.2,1) -- cycle;
\draw[red, ultra thick] (1,.8) -- (.2,.8);
\draw[red, ultra thick] (1,.6) -- (.2,.6);
\draw[red, ultra thick] (.4,.2) -- (.4,1);
\draw[red, ultra thick] (.8,.2) -- (.8,1);
\draw (0,0) -- (1.6,0) -- (1.6,1.2) -- (0,1.2) -- cycle;
\draw (.2,0) -- (.2,1.2);
\draw (.4,0) -- (.4, 1.2);
\draw (.6,0) -- (.6,.8);
\draw (.8,.2) -- (.8,1.2);
\draw (1,0) -- (1,1.2);
\draw (1.2,.2) -- (1.2,1);
\draw (1.4,0) -- (1.4,1.2);
\draw (0,.2) -- (1.6,.2);
\draw (.4,.4) -- (1.6,.4);
\draw (0,.6) -- (1.6,.6);
\draw (0,.8) -- (1.2,.8);
\draw (0,1) -- (1.6,1);
\end{tikzpicture}}+ \footnotesize{108}\cdot \raisebox{-1em}{\begin{tikzpicture}
\filldraw[draw=red, ultra thick, fill=red!50] (.4,.4) -- (1.2, .4) -- (1.2, 1) -- (.4,1) -- cycle;
\draw[red, ultra thick] (.4,.8) -- (1.2,.8);
\draw[red, ultra thick] (.4,.6) -- (1.2,.6);
\draw[red, ultra thick] (.8,.4) -- (.8,1);
\draw[red, ultra thick] (1,.4) -- (1,1);
\draw (0,0) -- (1.6,0) -- (1.6,1.2) -- (0,1.2) -- cycle;
\draw (.2,0) -- (.2,1.2);
\draw (.4,0) -- (.4, 1.2);
\draw (.6,0) -- (.6,.8);
\draw (.8,.2) -- (.8,1.2);
\draw (1,0) -- (1,1.2);
\draw (1.2,.2) -- (1.2,1);
\draw (1.4,0) -- (1.4,1.2);
\draw (0,.2) -- (1.6,.2);
\draw (.4,.4) -- (1.6,.4);
\draw (0,.6) -- (1.6,.6);
\draw (0,.8) -- (1.2,.8);
\draw (0,1) -- (1.6,1);
\end{tikzpicture}}\\&+  \footnotesize{324}\cdot \raisebox{-1em}{\begin{tikzpicture}
\filldraw[draw=red, ultra thick, fill=red!50] (.6,.2) -- (1.2, .2) -- (1.2, .8) -- (.6,.8) -- cycle;
\draw[red, ultra thick] (.6,.4) -- (1.2,.4);
\draw[red, ultra thick] (.6,.6) -- (1.2,.6);
\draw[red, ultra thick] (.8,.2) -- (.8,.8);
\draw[red, ultra thick] (1,.2) -- (1,.8);
\draw (0,0) -- (1.6,0) -- (1.6,1.2) -- (0,1.2) -- cycle;
\draw (.2,0) -- (.2,1.2);
\draw (.4,0) -- (.4, 1.2);
\draw (.6,0) -- (.6,.8);
\draw (.8,.2) -- (.8,1.2);
\draw (1,0) -- (1,1.2);
\draw (1.2,.2) -- (1.2,1);
\draw (1.4,0) -- (1.4,1.2);
\draw (0,.2) -- (1.6,.2);
\draw (.4,.4) -- (1.6,.4);
\draw (0,.6) -- (1.6,.6);
\draw (0,.8) -- (1.2,.8);
\draw (0,1) -- (1.6,1);
\end{tikzpicture}}+\footnotesize{360}\cdot \raisebox{-1em}{\begin{tikzpicture}
\filldraw[draw=red, ultra thick, fill=red!50] (.2,0) -- (1, 0) -- (1, .8) -- (.2,.8) -- cycle;
\draw[red, ultra thick] (.2,.2) -- (1,.2);
\draw[red, ultra thick] (.2,.6) -- (1,.6);
\draw[red, ultra thick] (.4,0) -- (.4,.8);
\draw[red, ultra thick] (.6,0) -- (.6,.8);
\draw (0,0) -- (1.6,0) -- (1.6,1.2) -- (0,1.2) -- cycle;
\draw (.2,0) -- (.2,1.2);
\draw (.4,0) -- (.4, 1.2);
\draw (.6,0) -- (.6,.8);
\draw (.8,.2) -- (.8,1.2);
\draw (1,0) -- (1,1.2);
\draw (1.2,.2) -- (1.2,1);
\draw (1.4,0) -- (1.4,1.2);
\draw (0,.2) -- (1.6,.2);
\draw (.4,.4) -- (1.6,.4);
\draw (0,.6) -- (1.6,.6);
\draw (0,.8) -- (1.2,.8);
\draw (0,1) -- (1.6,1);
\end{tikzpicture}}+  \footnotesize{108}\cdot \raisebox{-1em}{\begin{tikzpicture}
\filldraw[draw=red, ultra thick, fill=red!50] (.4,0) -- (1.4, 0) -- (1.4, .6) -- (.4,.6) -- cycle;
\draw[red, ultra thick] (.4,.2) -- (1.4,.2);
\draw[red, ultra thick] (.4,.4) -- (1.4,.4);
\draw[red, ultra thick] (.6,0) -- (.6,.6);
\draw[red, ultra thick] (1,0) -- (1,.6);
\draw (0,0) -- (1.6,0) -- (1.6,1.2) -- (0,1.2) -- cycle;
\draw (.2,0) -- (.2,1.2);
\draw (.4,0) -- (.4, 1.2);
\draw (.6,0) -- (.6,.8);
\draw (.8,.2) -- (.8,1.2);
\draw (1,0) -- (1,1.2);
\draw (1.2,.2) -- (1.2,1);
\draw (1.4,0) -- (1.4,1.2);
\draw (0,.2) -- (1.6,.2);
\draw (.4,.4) -- (1.6,.4);
\draw (0,.6) -- (1.6,.6);
\draw (0,.8) -- (1.2,.8);
\draw (0,1) -- (1.6,1);
\end{tikzpicture}}\\&+\footnotesize{268}\cdot \raisebox{-1em}{\begin{tikzpicture}
\filldraw[draw=red, ultra thick, fill=red!50] (.4,.2) -- (1, .2) -- (1, .8) -- (.4,.8) -- cycle;
\draw[red, ultra thick] (.4,.4) -- (1,.4);
\draw[red, ultra thick] (.4,.6) -- (1,.6);
\draw[red, ultra thick] (.6,.2) -- (.6,.8);
\draw[red, ultra thick] (.8,.2) -- (.8,.8);
\draw (0,0) -- (1.6,0) -- (1.6,1.2) -- (0,1.2) -- cycle;
\draw (.2,0) -- (.2,1.2);
\draw (.4,0) -- (.4, 1.2);
\draw (.6,0) -- (.6,.8);
\draw (.8,.2) -- (.8,1.2);
\draw (1,0) -- (1,1.2);
\draw (1.2,.2) -- (1.2,1);
\draw (1.4,0) -- (1.4,1.2);
\draw (0,.2) -- (1.6,.2);
\draw (.4,.4) -- (1.6,.4);
\draw (0,.6) -- (1.6,.6);
\draw (0,.8) -- (1.2,.8);
\draw (0,1) -- (1.6,1);
\end{tikzpicture}}
    \end{split}
\end{equation*}
\end{minipage}
}
\caption{Example of linear dependence in the LR B-spline set. The parameterization of an LR-mesh $\cM$ of multiplicity $1$ is considered in (a), and the linear dependence relation among some of the LR B-splines of bidegree $(2,2)$ defined on $\cM$ is illustrated in (b). The LR B-splines are represented by means of their supports on the mesh and the tensor meshes induced by their knots are highlighted with thicker meshlines.}\label{exLD}
\end{figure}

Unfortunately, they are not always linearly independent. Figure~\ref{exLD} shows an example of linear dependence among the LR B-splines of bidegree $(2,2)$ defined on an LR-mesh of multiplicity $1$.
To this day, it is not yet known what are the precise conditions on the LR-mesh to ensure a linearly independent set of LR B-splines.

In \cite{meshbressan} a characterization of the local linear independence of LR B-splines has been provided. Such a strong property is guaranteed only on LR-meshes with strong constraints on the split lengths and positions that yield particular arrangements of the LR B-spline supports. %Namely, the splits in it have to be such that no LR B-spline defined on the mesh has support fully contained in the support of another LR B-spline defined on the mesh.
This last statement is formalized in the following.
\begin{definition}\label{nesteddef}
Given a mesh $\cM$, let $B[\pmb{x}^1, \pmb{y}^1]$ and $B[\pmb{x}^2,\pmb{y}^2]$ be two different LR B-splines defined on $\cM$. We say that $B[\pmb{x}^2, \pmb{y}^2]$ is \highlight{nested} in $B[\pmb{x}^1, \pmb{y}^1]$, and we write $B[\pmb{x}^2, \pmb{y}^2] \preceq B[\pmb{x}^1, \pmb{y}^1]$, if
\begin{itemize}
\item $\supp B[\pmb{x}^2, \pmb{y}^2] \subseteq \supp B[\pmb{x}^1, \pmb{y}^1]$, and
\item any meshline $\gamma$ of $\cM$ in $\partial \supp B[\pmb{x}^1, \pmb{y}^1] \cap \partial \supp B[\pmb{x}^2, \pmb{y}^2]$ has a higher (or equal) multiplicity when considered in $\cM(\pmb{x}^1, \pmb{y}^1)$ than in $\cM(\pmb{x}^2, \pmb{y}^2)$.
\end{itemize}
An open mesh where no LR B-spline is nested is said to have the \highlight{non-nested support property}, or in short the \highlight{\NtwoS-property}.
\end{definition}

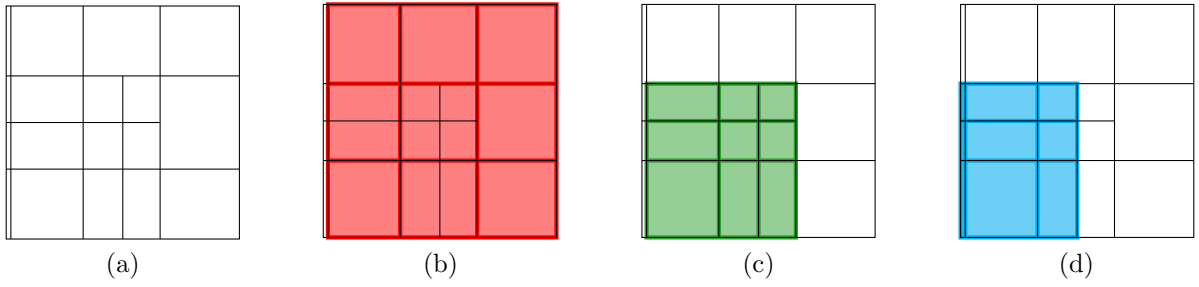
\begin{figure}[t!]
\centering
\subfloat[]{
\begin{tikzpicture}[scale=3.1]
\draw (0,0) -- (1,0) -- (1,1) -- (0,1) -- cycle;
\draw (.02,0) -- (.02,1);
\draw (.33, 0) -- (.33, 1);
\draw (.66, 0) -- (.66, 1);
\draw (0, .3) -- (1, .3);
\draw (0,.7) -- (1, .7);
\draw (.5, 0) -- (.5,.7);
\draw (0,.5) -- (.66,.5);
\end{tikzpicture}
}\qquad
\subfloat[]{
\begin{tikzpicture}[scale=3.1]
\filldraw[fill=red!50, draw = red, ultra thick] (.02,0) -- (1,0) -- (1,1) -- (.02,1) -- cycle;
\draw[red, ultra thick] (.02,.33) -- (1,.33);
\draw[red, ultra thick] (.02,.66) -- (1,.66);
\draw[red, ultra thick] (.33,0) -- (.33,1);
\draw[red, ultra thick] (.66,0) -- (.66,1);
\draw (0,0) -- (1,0) -- (1,1) -- (0,1) -- cycle;
\draw (.02,0) -- (.02,1);
\draw (.33, 0) -- (.33, 1);
\draw (.66, 0) -- (.66, 1);
\draw (0, .33) -- (1, .33);
\draw (0,.66) -- (1, .66);
\draw (.5, 0) -- (.5,.66);
\draw (0,.5) -- (.66,.5);
\end{tikzpicture}
}\qquad
\subfloat[]{
\begin{tikzpicture}[scale=3.1]
\filldraw[fill=green(ryb)!50, draw = green(ryb), ultra thick] (.02,0) -- (.66,0) -- (.66,.66) -- (.02,.66) -- cycle;
\draw[green(ryb), ultra thick] (.02,.33) -- (.66,.33);
\draw[green(ryb), ultra thick] (.02,.5) -- (.66,.5);
\draw[green(ryb), ultra thick] (.5,0) -- (.5,.66);
\draw[green(ryb), ultra thick] (.33,0) -- (.33,.66);
\draw (0,0) -- (1,0) -- (1,1) -- (0,1) -- cycle;
\draw (.02,0) -- (.02,1);
\draw (.33, 0) -- (.33, 1);
\draw (.66, 0) -- (.66, 1);
\draw (0, .33) -- (1, .33);
\draw (0,.66) -- (1, .66);
\draw (.5, 0) -- (.5,.66);
\draw (0,.5) -- (.66,.5);
\end{tikzpicture}
}\qquad
\subfloat[]{
\begin{tikzpicture}[scale=3.1]
\filldraw[fill=cyan!50, draw = cyan, ultra thick] (0,0) -- (.5,0) -- (.5,.66) -- (0,.66) -- cycle;
\draw[cyan, ultra thick] (0,.33) -- (.5,.33);
\draw[cyan, ultra thick] (0,.5) -- (.5,.5);
\draw[cyan, ultra thick] (.33,0) -- (.33,.66);
\draw[cyan, ultra thick] (.02,0) -- (.02,.66);
\draw (0,0) -- (1,0) -- (1,1) -- (0,1) -- cycle;
\draw (.02,0) -- (.02,1);
\draw (.33, 0) -- (.33, 1);
\draw (.66, 0) -- (.66, 1);
\draw (0, .33) -- (1, .33);
\draw (0,.66) -- (1, .66);
\draw (.5, 0) -- (.5,.66);
\draw (0,.5) -- (.66,.5);
\end{tikzpicture}
}\caption{Example of nested LR B-splines on the mesh $\cM$ shown in (a). All the meshlines have multiplicity~1 except those in the left edge of $\cM$, highlighted with a double line, which have multiplicity 2.  In (b)--(d) three LR B-splines, $B[\pmb{x}^1, \pmb{y}^1], B[\pmb{x}^2, \pmb{y}^2], B[\pmb{x}^3, \pmb{y}^3]$ respectively, of bidegree $(2,2)$ with minimal support on $\cM$ are represented by means of their supports and the tensor meshes induced by their knots. All the knots of these LR B-splines have multiplicity 1 except $x_1^3$ which has multiplicity 2. Therefore, $B[\pmb{x}^2, \pmb{y}^2] \preceq B[\pmb{x}^1, \pmb{y}^1]$ but $B[\pmb{x}^3, \pmb{y}^3] \npreceq B[\pmb{x}^2, \pmb{y}^2]$ and $B[\pmb{x}^3, \pmb{y}^3] \npreceq B[\pmb{x}^1, \pmb{y}^1]$, despite that $\supp B[\pmb{x}^3, \pmb{y}^3] \subseteq \supp B[\pmb{x}^2, \pmb{y}^2]$ and $\supp B[\pmb{x}^3, \pmb{y}^3] \subseteq\supp B[\pmb{x}^1, \pmb{y}^1]$, because the shared meshlines in the left edge of $\supp B[\pmb{x}^3, \pmb{y}^3]$, $\supp B[\pmb{x}^2, \pmb{y}^2]$ and $\supp B[\pmb{x}^1, \pmb{y}^1]$ have multiplicity 2 in $\cM(\pmb{x}^3, \pmb{y}^3)$ and multiplicity 1 in $\cM(\pmb{x}^2, \pmb{y}^2)$ and $\cM(\pmb{x}^1,\pmb{y}^1)$.}\label{nestedex}
\end{figure}

The definition of nested LR B-splines was formulated for the first time in \cite{someproperties}. Definition~\ref{nesteddef} is different but equivalent to it (see Appendix~\ref{sec:appendix}).
Figure~\ref{nestedex} shows an example of an LR B-spline nested into another.
%
% Note that for meshes of constant multiplicity, e.g., of multiplicity 1,  $B[\pmb{x}^2, \pmb{y}^2] \preceq B[\pmb{x}^1, \pmb{y}^1]$ if and only if $\supp B[\pmb{x}^2, \pmb{y}^2] \subseteq \supp B[\pmb{x}^1, \pmb{y}^1]$.
%
The following result, presented in \cite{meshbressan}, relates the local linear independence of the LR B-splines to the \NtwoS-property of the mesh.
\begin{thm}%[\cite{meshbressan}]
\label{TeoBressan}
Given a bidegree $\pmb{p}=(p_1,p_2)$, let $\cM$ be an open LR-mesh corresponding to a box-partition $\mathcal{E}$ and let $\mathcal{B}^{\mathcal{LR}}(\cM)$ be the set of LR B-splines of bidegree $\pmb{p}$ on $\cM$. The following statements are equivalent.
\begin{enumerate}
\item The elements of $\mathcal{B}^{\mathcal{LR}}(\cM)$ are locally linearly independent.
\item $\cM$ has the \NtwoS-property.
\item\label{TeoBressan:c} For any element $\beta \in \mathcal{E}$, the number of nonzero LR B-splines over $\beta$ satisfies
%$$\#\{B \in \mathcal{B}^{\mathcal{LR}}(\cM) : \supp B\supseteq \mathring{\beta}\} = (p_1+1)(p_2+1).$$
$$\#\{B \in \mathcal{B}^{\mathcal{LR}}(\cM) : \supp B\supseteq \beta\} = (p_1+1)(p_2+1).$$
\item The LR B-splines form a partition of unity, without the use of scaling weights.
\end{enumerate}
\end{thm}
An element of $\mathcal{E}$ for which item~\ref{TeoBressan:c} of Theorem~\ref{TeoBressan} holds is said to be \highlight{non-overloaded}. Note that $(p_1+1)(p_2+1)$ is the dimension of the polynomial space over the element.

In \cite{meshbressan} one can also find an algorithm to construct LR-meshes so that the \NtwoS-property is fulfilled. This approach, however, has a relevant drawback for practical purposes: the regions to be refined and the maximal resolution have to be chosen a priori. Moreover, the algorithm cannot be stopped prematurely, before having inserted all the splits determined initially.
In practice, one rarely knows in advance where the error will be large and how fine the mesh has to be chosen to reduce it under a certain tolerance.

In the next section, we present an alternative way to generate LR-meshes so that the \NtwoS-property is guaranteed.

%------------------------------------------------------------------------

\section{\NtwoS-structured mesh refinement strategy}\label{sec:refinement}
In this section, we define a local refinement strategy that ensures the \NtwoS-property for the obtained meshes. It consists of two steps. First, we apply the so-called structured mesh refinement, defined in \cite{johannessen}, to the LR B-splines whose contribution to the approximation error is larger than a given tolerance. Then, we slightly modify the obtained mesh by prolonging some splits, to recover the \NtwoS-property.
The meshes generated by this refinement are open meshes with internal meshlines of multiplicity one.

As opposed to the classical finite element method, in which the refinement is applied to the box-partition elements, the \highlight{structured mesh refinement} is a refinement applied to the function space, i.e., we select for refinement the LR B-splines contributing more to the approximation error rather than the box-partition elements where a larger error occurs.
This approach is justified by the fact that on an LR-mesh, any new split inserted must traverse at least the support of one LR B-spline. If we choose to select the elements where the error is larger, then the refinement has to be extended anyway to traverse the support of at least one LR B-spline containing the elements, resulting in a refinement of the LR B-spline basis. Moreover, since several LR B-splines contain such elements, those chosen for the refinement extension could be not those contributing more to the error, resulting in a suboptimal refinement, or we could refine more LR B-splines than necessary, wasting degrees of freedom.

%Given a spline approximation on the LR-mesh $\cM$, and a way to compute its error on each box-partition element, we measure the contribution of an LR B-spline to the approximation error by the so-called B-spline error, defined in \cite{johannessen}, as follows.
%\begin{definition} Given an LR-mesh $\cM$ corresponding to a box-partition $\cE$ and an LR B-spline $B = B[\pmb{x}, \pmb{y}]$ on $\cM$, let $\cE(B)\subseteq \cE$ be the subset of box-partition elements contained in $\supp B$. We define the \highlight{B-spline error} of $B$ as
%$$
%\norm e \norm_{\cE(B)}^2 := \sum_{\beta \in \cE(B)}  \norm e\norm_{\beta}^2,
%$$
%where $\norm e\norm_{\beta}$ is any fixed norm of the approximation error $e$ over $\beta \in \cE$.
%\end{definition}
Once the LR B-splines are selected, we refine them by halving the interval steps in their knot vectors. This corresponds to the insertion of a net of meshlines in the B-spline supports on the mesh. We therefore perform the LR B-spline generation algorithm described in Definition~\ref{defLR}. Every selected LR B-spline is fragmented into LR B-splines with smaller support and replaced by them. The LR-mesh obtained in this way will be called a \highlight{structured LR-mesh}.

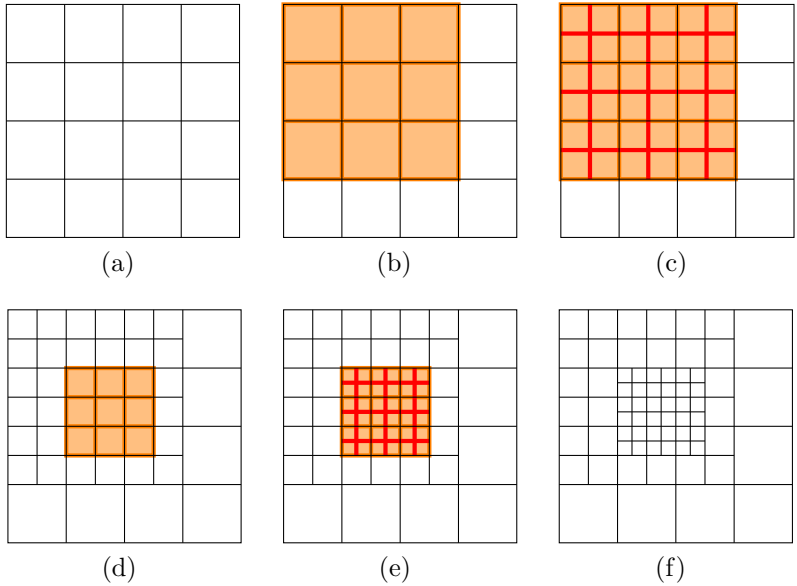
\begin{figure}[t!]
 \centering
 \subfloat[]{
 \begin{tikzpicture}[scale=3.1]
 \draw (0,0) -- (1,0) -- (1,1) -- (0,1) -- cycle;
 \draw (0,.25) -- (1,.25);
 \draw (0,.5) -- (1,.5);
 \draw (0,.75) -- (1,.75);
 \draw (.25,0) -- (.25,1);
 \draw (.5,0) -- (.5,1);
 \draw (.75,0) -- (.75,1);
 \end{tikzpicture}}\quad
 \subfloat[]{
 \begin{tikzpicture}[scale=3.1]
 \filldraw[draw = orange, ultra thick, fill=orange!50] (0,.25) -- (.75,.25) -- (.75,1) -- (0,1) -- cycle;
 \draw[orange, ultra thick] (0,.5) -- (.75,.5);
 \draw[orange, ultra thick] (0,.75) -- (.75,.75);
 \draw[orange, ultra thick] (.25,.25) -- (.25,1);
 \draw[orange, ultra thick] (.5,.25) -- (.5,1);
 \draw (0,0) -- (1,0) -- (1,1) -- (0,1) -- cycle;
 \draw (0,.25) -- (1,.25);
 \draw (0,.5) -- (1,.5);
 \draw (0,.75) -- (1,.75);
 \draw (.25,0) -- (.25,1);
 \draw (.5,0) -- (.5,1);
 \draw (.75,0) -- (.75,1);
 \end{tikzpicture}}\quad
  \subfloat[]{
 \begin{tikzpicture}[scale=3.1]
 \filldraw[draw = orange, ultra thick, fill=orange!50] (0,.25) -- (.75,.25) -- (.75,1) -- (0,1) -- cycle;
 \draw[orange, ultra thick] (0,.5) -- (.75,.5);
 \draw[orange, ultra thick] (0,.75) -- (.75,.75);
 \draw[orange, ultra thick] (.25,.25) -- (.25,1);
 \draw[orange, ultra thick] (.5,.25) -- (.5,1);
 \draw[red, ultra thick] (.125,.25) -- (.125,1);
 \draw[red, ultra thick] (.375,.25) -- (.375,1);
 \draw[red, ultra thick] (.625,.25) -- (.625,1);
 \draw[red, ultra thick] (0,.875) -- (.75,.875);
 \draw[red, ultra thick] (0,.375) -- (.75,.375);
 \draw[red, ultra thick] (0,.625) -- (.75,.625);
 \draw (0,0) -- (1,0) -- (1,1) -- (0,1) -- cycle;
 \draw (0,.25) -- (1,.25);
 \draw (0,.5) -- (1,.5);
 \draw (0,.75) -- (1,.75);
 \draw (.25,0) -- (.25,1);
 \draw (.5,0) -- (.5,1);
 \draw (.75,0) -- (.75,1);
 \end{tikzpicture}}\\
 \subfloat[]{
 \begin{tikzpicture}[scale=3.1]
 \filldraw[draw = orange, ultra thick, fill=orange!50] (.25,.375) -- (.625,.375) -- (.625,.75) -- (.25,.75) -- cycle;
 \draw[orange, ultra thick] (.25,.5) -- (.625,.5);
 \draw[orange, ultra thick] (.25,.625) -- (.625,.625);
 \draw[orange, ultra thick] (.375,.375) -- (.375,.75);
 \draw[orange, ultra thick] (.5,.375) -- (.5,.75);
  \draw (.125,.25) -- (.125,1);
 \draw (.375,.25) -- (.375,1);
 \draw (.625,.25) -- (.625,1);
 \draw (0,.875) -- (.75,.875);
 \draw (0,.375) -- (.75,.375);
 \draw (0,.625) -- (.75,.625);
 \draw (0,0) -- (1,0) -- (1,1) -- (0,1) -- cycle;
 \draw (0,.25) -- (1,.25);
 \draw (0,.5) -- (1,.5);
 \draw (0,.75) -- (1,.75);
 \draw (.25,0) -- (.25,1);
 \draw (.5,0) -- (.5,1);
 \draw (.75,0) -- (.75,1);
 \end{tikzpicture}}\quad
 \subfloat[]{
 \begin{tikzpicture}[scale=3.1]
 \filldraw[draw = orange, ultra thick, fill=orange!50] (.25,.375) -- (.625,.375) -- (.625,.75) -- (.25,.75) -- cycle;
 \draw[orange, ultra thick] (.25,.5) -- (.625,.5);
 \draw[orange, ultra thick] (.25,.625) -- (.625,.625);
 \draw[orange, ultra thick] (.375,.375) -- (.375,.75);
 \draw[orange, ultra thick] (.5,.375) -- (.5,.75);
 \draw[red, ultra thick] (.312,.375) -- (.312,.75);
 \draw[red, ultra thick] (.437,.375) -- (.437,.75);
 \draw[red, ultra thick] (.562,.375) -- (.562,.75);
 \draw[red, ultra thick] (.25,.437) -- (.625,.437);
 \draw[red, ultra thick] (.25,.562) -- (.625,.562);
   \draw[red, ultra thick] (.25,.687) -- (.625,.687);
 \draw (.125,.25) -- (.125,1);
 \draw (.375,.25) -- (.375,1);
 \draw (.625,.25) -- (.625,1);
 \draw (0,.875) -- (.75,.875);
 \draw (0,.375) -- (.75,.375);
 \draw (0,.625) -- (.75,.625);
 \draw (0,0) -- (1,0) -- (1,1) -- (0,1) -- cycle;
 \draw (0,.25) -- (1,.25);
 \draw (0,.5) -- (1,.5);
 \draw (0,.75) -- (1,.75);
 \draw (.25,0) -- (.25,1);
 \draw (.5,0) -- (.5,1);
 \draw (.75,0) -- (.75,1);
 \end{tikzpicture}}\quad
  \subfloat[]{
 \begin{tikzpicture}[scale=3.1]
 \draw (.312,.375) -- (.312,.75);
 \draw (.437,.375) -- (.437,.75);
 \draw (.562,.375) -- (.562,.75);
 \draw (.25,.437) -- (.625,.437);
 \draw (.25,.562) -- (.625,.562);
   \draw (.25,.687) -- (.625,.687);
 \draw (.125,.25) -- (.125,1);
 \draw (.375,.25) -- (.375,1);
 \draw (.625,.25) -- (.625,1);
 \draw (0,.875) -- (.75,.875);
 \draw (0,.375) -- (.75,.375);
 \draw (0,.625) -- (.75,.625);
 \draw (0,0) -- (1,0) -- (1,1) -- (0,1) -- cycle;
 \draw (0,.25) -- (1,.25);
 \draw (0,.5) -- (1,.5);
 \draw (0,.75) -- (1,.75);
 \draw (.25,0) -- (.25,1);
 \draw (.5,0) -- (.5,1);
 \draw (.75,0) -- (.75,1);
 \end{tikzpicture}}
 \caption{Two iterations of the structured mesh refinement of bidegree $(2,2)$. We consider the initial open tensor mesh with internal meshlines of multiplicity 1 in (a). Figure~(b) shows the support of an LR B-spline selected for refinement. We refine it by halving the interval steps in its knot vectors. This results in the insertion of a net of meshlines in the LR-mesh as shown in (c). In (d) we select another LR B-spline in the new set of LR B-splines  and we refine it as illustrated in (e). Figure~(f) depicts the final mesh obtained.
}\label{exSM}
\end{figure}
In summary, the structured mesh refinement consists of two steps:
\begin{enumerate}
  \item LR B-splines are selected to be refined and not box-partition elements;
  \item the interval steps of their knot vectors are halved to obtain the new LR-mesh.
\end{enumerate}
Figure~\ref{exSM} shows two iterations of such refinement.
In general, the structured mesh refinement does not generate LR-meshes with the \NtwoS-property. The LR-mesh in Figure~\ref{exSM}(f) is an example as explained in Figure~\ref{exSM2}.
Furthermore, the structured mesh refinement may produce linearly dependent sets of LR B-splines. Figure~\ref{LDstructured} shows an example for bidegree $(4,4)$.

\begin{figure}[t!]
 \centering
 \subfloat[]{
 \begin{tikzpicture}[scale=3.1]
 \filldraw[draw=blue, ultra thick, fill =blue!50] (.437,.375) -- (.625,.375) -- (.625,.562) -- (.437,.562) -- cycle;
 \draw[blue, ultra thick] (.437, .437) -- (.625,.437);
 \draw[blue, ultra thick] (.437,.5) -- (.625,.5);
 \draw[blue, ultra thick] (.5, .375) -- (.5,.562);
 \draw[blue, ultra thick] (.562,.375) -- (.562,.562);
 \draw (.312,.375) -- (.312,.75);
 \draw (.437,.375) -- (.437,.75);
 \draw (.562,.375) -- (.562,.75);
 \draw (.25,.437) -- (.625,.437);
 \draw (.25,.562) -- (.625,.562);
   \draw (.25,.687) -- (.625,.687);
 \draw (.125,.25) -- (.125,1);
 \draw (.375,.25) -- (.375,1);
 \draw (.625,.25) -- (.625,1);
 \draw (0,.875) -- (.75,.875);
 \draw (0,.375) -- (.75,.375);
 \draw (0,.625) -- (.75,.625);
 \draw (0,0) -- (1,0) -- (1,1) -- (0,1) -- cycle;
 \draw (0,.25) -- (1,.25);
 \draw (0,.5) -- (1,.5);
 \draw (0,.75) -- (1,.75);
 \draw (.25,0) -- (.25,1);
 \draw (.5,0) -- (.5,1);
 \draw (.75,0) -- (.75,1);
 \end{tikzpicture}}\quad
  \subfloat[]{
 \begin{tikzpicture}[scale=3.1]
 \filldraw[draw=green(ryb), ultra thick, fill =green(ryb)!50] (.375,.25) -- (.75,.25) -- (.75,.625) -- (.375,.625) -- cycle;
 \draw[green(ryb), ultra thick] (.375, .375) -- (.75,.375);
 \draw[green(ryb), ultra thick] (.375,.5) -- (.75,.5);
 \draw[green(ryb), ultra thick] (.5, .25) -- (.5,.625);
 \draw[green(ryb), ultra thick] (.625,.25) -- (.625,.625);
 \draw (.312,.375) -- (.312,.75);
 \draw (.437,.375) -- (.437,.75);
 \draw (.562,.375) -- (.562,.75);
 \draw (.25,.437) -- (.625,.437);
 \draw (.25,.562) -- (.625,.562);
   \draw (.25,.687) -- (.625,.687);
 \draw (.125,.25) -- (.125,1);
 \draw (.375,.25) -- (.375,1);
 \draw (.625,.25) -- (.625,1);
 \draw (0,.875) -- (.75,.875);
 \draw (0,.375) -- (.75,.375);
 \draw (0,.625) -- (.75,.625);
 \draw (0,0) -- (1,0) -- (1,1) -- (0,1) -- cycle;
 \draw (0,.25) -- (1,.25);
 \draw (0,.5) -- (1,.5);
 \draw (0,.75) -- (1,.75);
 \draw (.25,0) -- (.25,1);
 \draw (.5,0) -- (.5,1);
 \draw (.75,0) -- (.75,1);
 \end{tikzpicture}}\quad
  \subfloat[]{
 \begin{tikzpicture}[scale=3.1]
  \filldraw[draw=red, ultra thick, fill =red!50] (.25,0) -- (1,0) -- (1,.75) -- (.25,.75) -- cycle;
 \draw[red, ultra thick] (.5, 0) -- (.5,.75);
 \draw[red, ultra thick] (.75,0) -- (.75,.75);
 \draw[red, ultra thick] (.25, .25) -- (1,.25);
 \draw[red, ultra thick] (.25,.5) -- (1,.5);
 \draw (.312,.375) -- (.312,.75);
 \draw (.437,.375) -- (.437,.75);
 \draw (.562,.375) -- (.562,.75);
 \draw (.25,.437) -- (.625,.437);
 \draw (.25,.562) -- (.625,.562);
   \draw (.25,.687) -- (.625,.687);
 \draw (.125,.25) -- (.125,1);
 \draw (.375,.25) -- (.375,1);
 \draw (.625,.25) -- (.625,1);
 \draw (0,.875) -- (.75,.875);
 \draw (0,.375) -- (.75,.375);
 \draw (0,.625) -- (.75,.625);
 \draw (0,0) -- (1,0) -- (1,1) -- (0,1) -- cycle;
 \draw (0,.25) -- (1,.25);
 \draw (0,.5) -- (1,.5);
 \draw (0,.75) -- (1,.75);
 \draw (.25,0) -- (.25,1);
 \draw (.5,0) -- (.5,1);
 \draw (.75,0) -- (.75,1);
 \end{tikzpicture}}\quad
  \subfloat[]{
 \begin{tikzpicture}[scale=3.1]
   \filldraw[draw=red, ultra thick, fill =red!50] (.25,0) -- (1,0) -- (1,.75) -- (.25,.75) -- cycle;
 \draw[red, ultra thick] (.5, 0) -- (.5,.75);
 \draw[red, ultra thick] (.75,0) -- (.75,.75);
 \draw[red, ultra thick] (.25, .25) -- (1,.25);
 \draw[red, ultra thick] (.25,.5) -- (1,.5);
  \filldraw[draw=green(ryb), ultra thick, fill =green(ryb)!50] (.375,.25) -- (.75,.25) -- (.75,.625) -- (.375,.625) -- cycle;
 \draw[green(ryb), ultra thick] (.375, .375) -- (.75,.375);
 \draw[green(ryb), ultra thick] (.375,.5) -- (.75,.5);
 \draw[green(ryb), ultra thick] (.5, .25) -- (.5,.625);
 \draw[green(ryb), ultra thick] (.625,.25) -- (.625,.625);
   \filldraw[draw=blue, ultra thick, fill =blue!50] (.437,.375) -- (.625,.375) -- (.625,.562) -- (.437,.562) -- cycle;
 \draw[blue, ultra thick] (.437, .437) -- (.625,.437);
 \draw[blue, ultra thick] (.437,.5) -- (.625,.5);
 \draw[blue, ultra thick] (.5, .375) -- (.5,.562);
 \draw[blue, ultra thick] (.562,.375) -- (.562,.562);
 \draw (.312,.375) -- (.312,.75);
 \draw (.437,.375) -- (.437,.75);
 \draw (.562,.375) -- (.562,.75);
 \draw (.25,.437) -- (.625,.437);
 \draw (.25,.562) -- (.625,.562);
   \draw (.25,.687) -- (.625,.687);
 \draw (.125,.25) -- (.125,1);
 \draw (.375,.25) -- (.375,1);
 \draw (.625,.25) -- (.625,1);
 \draw (0,.875) -- (.75,.875);
 \draw (0,.375) -- (.75,.375);
 \draw (0,.625) -- (.75,.625);
 \draw (0,0) -- (1,0) -- (1,1) -- (0,1) -- cycle;
 \draw (0,.25) -- (1,.25);
 \draw (0,.5) -- (1,.5);
 \draw (0,.75) -- (1,.75);
 \draw (.25,0) -- (.25,1);
 \draw (.5,0) -- (.5,1);
 \draw (.75,0) -- (.75,1);
 \end{tikzpicture}}
 \caption{An LR B-spline nested in another LR B-spline, which in turn is nested in another LR B-spline on a structured LR-mesh for bidegree $(2,2)$. Consider again the mesh in Figure~\ref{exSM}(f). In (a)--(c) we depict the supports of three LR B-splines on this mesh. The support in (a) is contained in the interior of the support in (b) and (c), and the support in (b) is contained in the interior of the support in (c). Therefore, the LR B-spline considered in (a) is nested both in the LR B-splines in (b) and (c), and the LR B-spline in (b) is nested in the LR B-spline in (c). Hence, the considered mesh does not have the \NtwoS-property.}\label{exSM2}
 \end{figure}
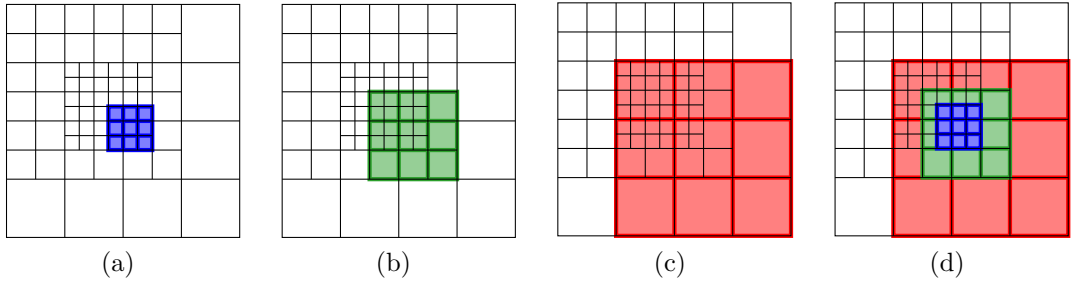
\begin{figure}[t!]
\centering
\subfloat[]{
\begin{tikzpicture}[scale=4.5]
\draw (0,0) -- (1,0) -- (1,1) -- (0,1) -- cycle;
\draw (0,.111) -- (1,.111);
\draw (0,.222) -- (1,.222);
\draw (0,.333) -- (1,.333);
\draw (0,.444) -- (1,.444);
\draw (0,.555) -- (1,.555);
\draw (0,.666) -- (1,.666);
\draw (0,.777) -- (1,.777);
\draw (0,.888) -- (1,.888);
\draw (.111,0) -- (.111,1);
\draw (.222,0) -- (.222,1);
\draw (.333,0) -- (.333,1);
\draw (.444,0) -- (.444,1);
\draw (.555,0) -- (.555,1);
\draw (.666,0) -- (.666,1);
\draw (.777,0) -- (.777,1);
\draw (.888,0) -- (.888,1);
\end{tikzpicture}
}\hspace*{0.2cm}
\subfloat[]{
\begin{tikzpicture}[scale=4.5]
\draw (0,0) -- (1,0) -- (1,1) -- (0,1) -- cycle;
\draw (0,.111) -- (1,.111);
\draw (0,.222) -- (1,.222);
\draw (0,.333) -- (1,.333);
\draw (0,.444) -- (1,.444);
\draw (0,.555) -- (1,.555);
\draw (0,.666) -- (1,.666);
\draw (0,.777) -- (1,.777);
\draw (0,.888) -- (1,.888);
\draw (.111,0) -- (.111,1);
\draw (.222,0) -- (.222,1);
\draw (.333,0) -- (.333,1);
\draw (.444,0) -- (.444,1);
\draw (.555,0) -- (.555,1);
\draw (.666,0) -- (.666,1);
\draw (.777,0) -- (.777,1);
\draw (.888,0) -- (.888,1);
\draw (0,.943) -- (.555,.943);
\draw (0,.833) -- (.555,.833);
\draw (0,.722) -- (.555,.722);
\draw (0,.611) -- (.555,.611);
\draw (0,.5) -- (1,.5);
\draw (.444,.389) -- (1,.389);
\draw (.444,.278) -- (1,.278);
\draw (.444,.167) -- (1,.167);
\draw (.444,.055) -- (1,.055);
\draw (.943,0) -- (.943,.555);
\draw (.833,0) -- (.833,.555);
\draw (.722,0) -- (.722,.555);
\draw (.611,0) -- (.611,.555);
\draw (.5,0) -- (.5,1);
\draw (.389,.444) -- (.389,1);
\draw (.278,.444) -- (.278,1);
\draw (.167,.444) -- (.167,1);
\draw (.055,.444) -- (.055,1);
\end{tikzpicture}
}\hspace*{0.2cm}
\subfloat[]{
\begin{tikzpicture}[scale=4.5]
\filldraw[fill=green(ryb)!50, draw=green(ryb), ultra thick] (.333,.222) -- (.777,.222) -- (.777,.666) -- (.333,.666) -- cycle;
\draw (0,0) -- (1,0) -- (1,1) -- (0,1) -- cycle;
\draw (0,.111) -- (1,.111);
\draw (0,.222) -- (1,.222);
\draw (0,.333) -- (1,.333);
\draw (0,.444) -- (1,.444);
\draw (0,.555) -- (1,.555);
\draw (0,.666) -- (1,.666);
\draw (0,.777) -- (1,.777);
\draw (0,.888) -- (1,.888);
\draw (.111,0) -- (.111,1);
\draw (.222,0) -- (.222,1);
\draw (.333,0) -- (.333,1);
\draw (.444,0) -- (.444,1);
\draw (.555,0) -- (.555,1);
\draw (.666,0) -- (.666,1);
\draw (.777,0) -- (.777,1);
\draw (.888,0) -- (.888,1);
\draw (0,.943) -- (.555,.943);
\draw (0,.833) -- (.555,.833);
\draw (0,.722) -- (.555,.722);
\draw (0,.611) -- (.555,.611);
\draw (0,.5) -- (1,.5);
\draw (.444,.389) -- (1,.389);
\draw (.444,.278) -- (1,.278);
\draw (.444,.167) -- (1,.167);
\draw (.444,.055) -- (1,.055);
\draw (.943,0) -- (.943,.555);
\draw (.833,0) -- (.833,.555);
\draw (.722,0) -- (.722,.555);
\draw (.611,0) -- (.611,.555);
\draw (.5,0) -- (.5,1);
\draw (.389,.444) -- (.389,1);
\draw (.278,.444) -- (.278,1);
\draw (.167,.444) -- (.167,1);
\draw (.055,.444) -- (.055,1);
\draw (.278,.693) -- (.555,.693);
\draw (.278,.638) -- (.555,.638);
\draw (.278,.584) -- (.555,.584);
\draw (.278,.528) -- (.722,.528);
\draw (.278,.473) -- (.722,.473);
\draw (.444,.417) -- (.722,.417);
\draw (.444,.362) -- (.722,.362);
\draw (.444,.306) -- (.722,.306);
\draw (.693,.278) -- (.693,.555);
\draw (.638,.278) -- (.638,.555);
\draw (.584,.278) -- (.584,.555);
\draw (.528,.278) -- (.528,.722);
\draw (.473,.278) -- (.473,.722);
\draw (.417,.444) -- (.417,.722);
\draw (.362,.444) -- (.362,.722);
\draw (.306,.444) -- (.306,.722);
\end{tikzpicture}
}
\caption{A structured LR-mesh with a linear dependence relation among the LR B-splines of bidegree $(4,4)$ defined on the highlighted region in (c). We start by considering an open tensor mesh with interior meshlines of multiplicity 1 as in (a). Then, we apply two iterations of structured mesh refinement as shown in (b)--(c). The LR B-splines with support in the region highlighted in (c) are linearly dependent. In particular, the region corresponds to the support of an LR B-spline that has many nested LR B-splines in it. One can prove the existence of the linear dependence relation by computing the spline space dimension and the number of LR B-splines defined on the mesh as explained in the examples of \cite{LRdependence}. This configuration can be reproduced for any bidegree $(p_1, p_2)$ with $p_k \geq 4$ for $k = 1,2$.}\label{LDstructured}
\end{figure}
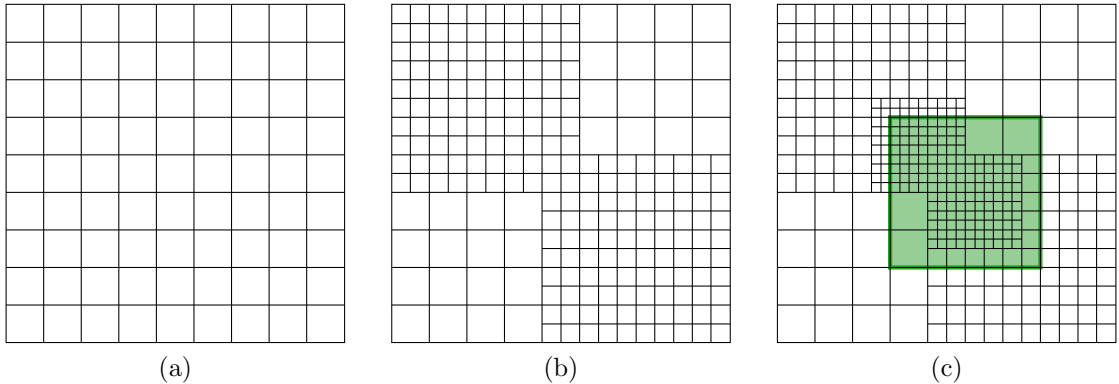

On the other hand, the standard B-splines defined on a plain tensor mesh are locally linearly independent, and the meshes generated by the structured mesh refinement are locally tensor meshes far from the boundary of the region where the structured mesh refinement is applied. The LR B-splines defined in these zones of the mesh behave like the standard B-splines, and therefore are locally linearly independent. On the boundary of the region where the refinement has been applied, LR B-splines with smaller support can be nested in LR B-splines with larger support. %These LR B-splines with larger support were already defined on coarser meshes and have not been refined by the structured mesh refinement strategy despite of their overlap with the LR B-splines invoking the larger B-spline error.
Hence, in such case the resulting LR-mesh does not have the \NtwoS-property.

The idea for our refinement strategy, which will be called \highlight{\NtwoS-structured mesh refinement}, is therefore to recover the \NtwoS-property in the mesh by slightly modifying it in these transition regions. When an LR B-spline $B[\pmb{x}^2, \pmb{y}^2]$ is nested into another LR B-spline $B[\pmb{x}^1, \pmb{y}^1]$, one could prolong the splits in $\cM(\pmb{x}^2, \pmb{y}^2)$ in some direction to traverse entirely $\supp B[\pmb{x}^1, \pmb{y}^1]$. This, by Definition \ref{defLR}, would refine $B[\pmb{x}^1, \pmb{y}^1]$ in LR B-splines that turn out not to have nested LR B-splines in their supports anymore.
% This last statement is formalized in Corollary \ref{lemmaTE}. To this end, we first need to introduce the orientation of T-vertices in a box-partition and prove the \NtwoS-property for LR-meshes with a particular structure.
% \begin{definition}
% Any T-vertex in a box-partition is the intersection of two collinear meshlines and another meshline, say $\gamma$, orthogonal to them.
% We call the T-vertex \highlight{vertical} if $\gamma$ is vertical,
% and \highlight{horizontal} otherwise.
% \end{definition}
This last statement is formalized in Corollary \ref{lemmaTE}. To this end, we first prove the \NtwoS-property for LR-meshes with a particular structure.
\begin{definition}
An LR-mesh $\cM$ on the domain $\Omega$ is said to be \highlight{tensorized in the $k$-th direction}, for $k \in \{1,2\}$, if all the internal $k$-meshlines in $\cM$ are contained in $k$-splits crossing $\Omega$ entirely, i.e., there are no vertical, if $k=1$, or horizontal, if $k=2$, T-vertices in the interior of $\Omega$.
\end{definition}
\begin{prop}\label{N2Stensorized}
Let $\cM$ be an LR-mesh tensorized in the $k$-th direction for some $k \in \{1,2\}$. Then, the LR B-splines defined on $\cM$ are all non-nested.
\end{prop}
\begin{proof}
Without loss of generality, we can assume that $\cM$ is tensorized in the first direction, i.e., the  vertical meshlines are all contained in vertical splits crossing the domain entirely.  This means that in $\cM$ no vertical meshline ends in the interior of the domain and therefore in the interior of the support of any LR B-splines defined on $\cM$. We now proceed by contradiction and assume that there exists an LR B-spline in $\cM$, say $B^2 = B[\pmb{x}^2, \pmb{y}^2]$, nested in another, say $B^1 = B[\pmb{x}^1, \pmb{y}^1]$; see Definition~\ref{nesteddef}. Because of the tensorization in the first direction, this can only happen if they share the same knot vector in the $x$-direction, $\pmb{x}^1= \pmb{x}^2$.
In particular, their supports have the same extreme values in the $x$-direction.
This implies that all the horizontal splits, counting the multiplicities, of $\cM$ traversing $\supp B^2$ must traverse $\supp B^1$ as well.
%Since both $B^1$ and $B^2$ are LR B-splines, they have minimal support on $\cM$. %and hence all these horizontal splits must be in both $\cM[\pmb{x}^1, \pmb{y}^1]$ and $\cM[\pmb{x}^2, \pmb{y}^2]$.
Since $B^2$ is nested in $B^1$ and $B^1$ has minimal support (it is an LR  B-spline), it follows that $\pmb{y}^1 = \pmb{y}^2$, and as a consequence we have $B^1=B^2$. This is a contradiction and concludes the proof.
%
% In other words, since $\supp B^2 \subseteq \supp B^1$, see Definition~\ref{nesteddef}, $\pmb{y}^1$ and $\pmb{y}^2$ have the same internal distinct knots with the same multiplicities. In order then to have $B^2\preceq B^1$, also the boundary horizontal meshlines in the tensor meshes defined by the knots of $B^1$ and $B^2$ respectively must coincide and have the same multiplicities. Indeed, if one of such multiplicities is lower in $\cM[\pmb{x}^1, \pmb{y}^1]$ than in $\cM[\pmb{x}^2, \pmb{y}^2]$, $B^2$ would not be nested in $B^1$ and, on the contrary, if a multiplicity is higher in $\cM[\pmb{x}^1, \pmb{y}^1]$ than in $\cM[\pmb{x}^2, \pmb{y}^2]$, $B^1$ would not have minimal support because $\supp B^2 \subseteq \supp B^1$. Therefore, $\pmb{y}^1 = \pmb{y}^2$, and $B^1, B^2$ must be the same function as they are defined on the same knot vectors. This is a contradiction and concludes the proof.
\end{proof}
\begin{cor}\label{lemmaTE}
Given an LR-mesh $\cM$, let $B = B[\pmb{x},\pmb{y}]$ and $B^1 = B[\pmb{x}^1, \pmb{y}^1],$ $\ldots,$ $B^n = B[\pmb{x}^n, \pmb{y}^n]$ be LR B-splines defined on $\cM$ such that $ B^1, \ldots,  B^n \preceq B$. Let $\cN$ be the mesh defined by the restriction of $\cM$ to the meshlines of $\cM(\pmb{x}, \pmb{y}), \cM(\pmb{x}^1, \pmb{y}^1), \ldots, \cM(\pmb{x}^n, \pmb{y}^n)$.
Then,\begin{enumerate}
\item there are at least one horizontal T-vertex and one vertical T-vertex of $\cN$ in the interior of $\supp B$;
\item by extending all the splits of $\cN$ in some direction to cross $\supp B$ entirely, $B$ is refined, by Definition \ref{defLR}, in LR B-splines that do not have any nested LR B-splines anymore.\label{lemmaTE:b}
\end{enumerate}
\end{cor}
\begin{proof}
\begin{enumerate}
\item Assume that there are no vertical T-vertices of $\cN$ in the interior of $\supp B$. Then, $\cN$ would be tensorized in the first direction. By Proposition~\ref{N2Stensorized}, it would imply that all the LR B-splines defined on $\cN$ are non-nested, which is a contradiction. Analogously, one can prove that at least one horizontal T-vertex of $\cN$ must be in the interior of $\supp B$.
\item Since $B$ contains the support of other B-splines, $\cN \neq \cM(\pmb{x}, \pmb{y})$ and in particular there exist at least one horizontal T-vertex and one vertical T-vertex by the previous item. We now focus on the vertical T-vertices, but of course the same argument can also be carried out for the horizontal T-vertices. We extend all the vertical splits in $\cN$ to cross $\supp B$ entirely, and denote this new mesh as $\widetilde{\cN}$. By Definition \ref{defLR}, the extensions trigger a refinement of $B$ via knot insertions. $\widetilde{\cN}$  is tensorized in the first direction and, by Proposition~\ref{N2Stensorized}, no LR B-spline defined on $\widetilde{\cN}$ is nested into another.%
\end{enumerate}
\end{proof}
The extension of the splits considered in item~\ref{lemmaTE:b} of Corollary~\ref{lemmaTE} will be called a \highlight{\unilateral{} tensor expansion} of $B^1, \ldots, B^n$ in $B$. An example is illustrated in Figure~\ref{exTensorEx}.~
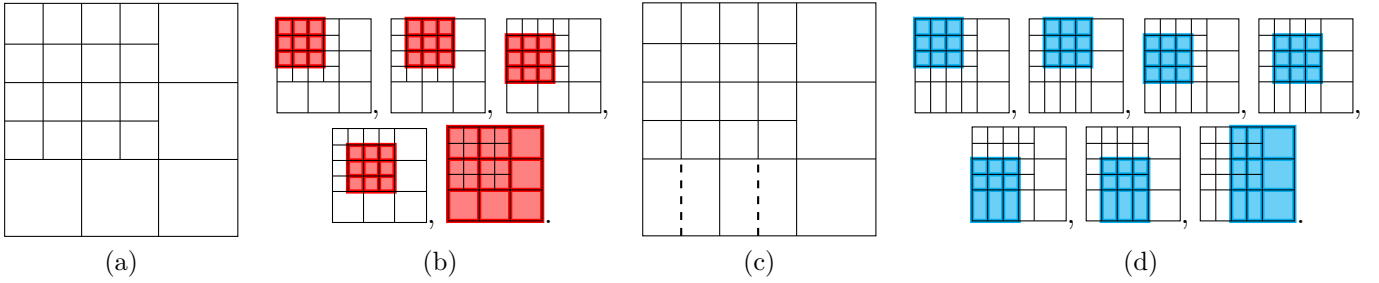
\begin{figure}[t!]\hspace*{-0.8cm}
%\centering
\subfloat[]{
\begin{tikzpicture}[scale=3.1]
\draw (0,0) -- (1,0) -- (1,1) -- (0,1) -- cycle;
\draw (.33,0) -- (.33, 1);
\draw (.66,0) -- (.66,1);
\draw (0,.66) -- (1,.66);
\draw (0,.33) -- (1,.33);
\draw (0,.495) -- (.66,.495);
\draw (0,.825) -- (.66,.825);
\draw (.495,.33) -- (.495,1);
\draw (.165,.33) -- (.165,1);
\end{tikzpicture}
}
\subfloat[]{
\begin{minipage}{.2\textwidth}\vspace{-2.75cm}
\begin{tabular}{c}
\begin{tikzpicture}[scale=1.25]
\filldraw[draw=red, ultra thick, fill=red!50] (0,.495) -- (.495,.495) -- (.495,1) -- (0,1) -- cycle;
\draw[red, ultra thick] (.165,.495) -- (.165,1);
\draw[red, ultra thick] (.33,.495) -- (.33,1);
\draw[red, ultra thick] (0,.66) -- (.495,.66);
\draw[red, ultra thick] (0,.825) -- (.495,.825);
\draw (0,0) -- (1,0) -- (1,1) -- (0,1) -- cycle;
\draw (.33,0) -- (.33, 1);
\draw (.66,0) -- (.66,1);
\draw (0,.66) -- (1,.66);
\draw (0,.33) -- (1,.33);
\draw (0,.495) -- (.66,.495);
\draw (0,.825) -- (.66,.825);
\draw (.495,.33) -- (.495,1);
\draw (.165,.33) -- (.165,1);
\end{tikzpicture},
\begin{tikzpicture}[scale=1.25]
\filldraw[draw=red, ultra thick, fill=red!50] (.165,.495) -- (.66,.495) -- (.66,1) -- (.165,1) -- cycle;
\draw[red, ultra thick] (.495,.495) -- (.495,1);
\draw[red, ultra thick] (.33,.495) -- (.33,1);
\draw[red, ultra thick] (.165,.66) -- (.66,.66);
\draw[red, ultra thick] (.165,.825) -- (.66,.825);
\draw (0,0) -- (1,0) -- (1,1) -- (0,1) -- cycle;
\draw (.33,0) -- (.33, 1);
\draw (.66,0) -- (.66,1);
\draw (0,.66) -- (1,.66);
\draw (0,.33) -- (1,.33);
\draw (0,.495) -- (.66,.495);
\draw (0,.825) -- (.66,.825);
\draw (.495,.33) -- (.495,1);
\draw (.165,.33) -- (.165,1);
\end{tikzpicture},
\begin{tikzpicture}[scale=1.25]
\filldraw[draw=red, ultra thick, fill=red!50] (0,.33) -- (.495,.33) -- (.495,.825) -- (0,.825) -- cycle;
\draw[red, ultra thick] (.165,.33) -- (.165,.825);
\draw[red, ultra thick] (.33,.33) -- (.33,.825);
\draw[red, ultra thick] (0,.66) -- (.495,.66);
\draw[red, ultra thick] (0,.495) -- (.495,.495);
\draw (0,0) -- (1,0) -- (1,1) -- (0,1) -- cycle;
\draw (.33,0) -- (.33, 1);
\draw (.66,0) -- (.66,1);
\draw (0,.66) -- (1,.66);
\draw (0,.33) -- (1,.33);
\draw (0,.495) -- (.66,.495);
\draw (0,.825) -- (.66,.825);
\draw (.495,.33) -- (.495,1);
\draw (.165,.33) -- (.165,1);
\end{tikzpicture},\\
\begin{tikzpicture}[scale=1.25]
\filldraw[draw=red, ultra thick, fill=red!50] (.165,.33) -- (.66,.33) -- (.66,.825) -- (.165,.825) -- cycle;
\draw[red, ultra thick] (.495,.33) -- (.495,.825);
\draw[red, ultra thick] (.33,.33) -- (.33,.825);
\draw[red, ultra thick] (.165,.66) -- (.66,.66);
\draw[red, ultra thick] (.165,.495) -- (.66,.495);
\draw (0,0) -- (1,0) -- (1,1) -- (0,1) -- cycle;
\draw (.33,0) -- (.33, 1);
\draw (.66,0) -- (.66,1);
\draw (0,.66) -- (1,.66);
\draw (0,.33) -- (1,.33);
\draw (0,.495) -- (.66,.495);
\draw (0,.825) -- (.66,.825);
\draw (.495,.33) -- (.495,1);
\draw (.165,.33) -- (.165,1);
\end{tikzpicture},
\begin{tikzpicture}[scale=1.25]
\filldraw[draw=red, ultra thick, fill=red!50] (0,0) -- (1,0) -- (1,1) -- (0,1) -- cycle;
\draw[red, ultra thick] (.66,0) -- (.66,1);
\draw[red, ultra thick] (.33,0) -- (.33,1);
\draw[red, ultra thick] (0,.33) -- (1,.33);
\draw[red, ultra thick] (0,.66) -- (1,.66);
\draw (0,0) -- (1,0) -- (1,1) -- (0,1) -- cycle;
\draw (.33,0) -- (.33, 1);
\draw (.66,0) -- (.66,1);
\draw (0,.66) -- (1,.66);
\draw (0,.33) -- (1,.33);
\draw (0,.495) -- (.66,.495);
\draw (0,.825) -- (.66,.825);
\draw (.495,.33) -- (.495,1);
\draw (.165,.33) -- (.165,1);
\end{tikzpicture}.
\end{tabular}
\end{minipage}\hspace{1.5cm}
}
\subfloat[]{
\begin{tikzpicture}[scale=3.1]
\draw (0,0) -- (1,0) -- (1,1) -- (0,1) -- cycle;
\draw (.33,0) -- (.33, 1);
\draw (.66,0) -- (.66,1);
\draw (0,.66) -- (1,.66);
\draw (0,.33) -- (1,.33);
\draw (0,.495) -- (.66,.495);
\draw (0,.825) -- (.66,.825);
\draw (.495,.33) -- (.495,1);
\draw (.165,.33) -- (.165,1);
\draw[thick, dashed] (.165,0) -- (.165,.33);
\draw[thick, dashed] (.495,0) -- (.495,.33);
\end{tikzpicture}
}
\subfloat[]{
\begin{minipage}{.3\textwidth}\vspace{-2.75cm}
\begin{tabular}{c}
\begin{tikzpicture}[scale=1.25]
\filldraw[draw=cyan, ultra thick, fill=cyan!50] (0,.495) -- (.495,.495) -- (.495,1) -- (0,1) -- cycle;
\draw[cyan, ultra thick] (.165,.495) -- (.165,1);
\draw[cyan, ultra thick] (.33,.495) -- (.33,1);
\draw[cyan, ultra thick] (0,.66) -- (.495,.66);
\draw[cyan, ultra thick] (0,.825) -- (.495,.825);
\draw (0,0) -- (1,0) -- (1,1) -- (0,1) -- cycle;
\draw (.33,0) -- (.33, 1);
\draw (.66,0) -- (.66,1);
\draw (0,.66) -- (1,.66);
\draw (0,.33) -- (1,.33);
\draw (0,.495) -- (.66,.495);
\draw (0,.825) -- (.66,.825);
\draw (.495,.33) -- (.495,1);
\draw (.165,.33) -- (.165,1);
\draw (.165,0) -- (.165,.33);
\draw (.495,0) -- (.495,.33);
\end{tikzpicture},
\begin{tikzpicture}[scale=1.25]
\filldraw[draw=cyan, ultra thick, fill=cyan!50] (.165,.495) -- (.66,.495) -- (.66,1) -- (.165,1) -- cycle;
\draw[cyan, ultra thick] (.495,.495) -- (.495,1);
\draw[cyan, ultra thick] (.33,.495) -- (.33,1);
\draw[cyan, ultra thick] (.165,.66) -- (.66,.66);
\draw[cyan, ultra thick] (.165,.825) -- (.66,.825);
\draw (0,0) -- (1,0) -- (1,1) -- (0,1) -- cycle;
\draw (.33,0) -- (.33, 1);
\draw (.66,0) -- (.66,1);
\draw (0,.66) -- (1,.66);
\draw (0,.33) -- (1,.33);
\draw (0,.495) -- (.66,.495);
\draw (0,.825) -- (.66,.825);
\draw (.495,.33) -- (.495,1);
\draw (.165,.33) -- (.165,1);
\draw (.165,0) -- (.165,.33);
\draw (.495,0) -- (.495,.33);
\end{tikzpicture},
\begin{tikzpicture}[scale=1.25]
\filldraw[draw=cyan, ultra thick, fill=cyan!50] (0,.33) -- (.495,.33) -- (.495,.825) -- (0,.825) -- cycle;
\draw[cyan, ultra thick] (.165,.33) -- (.165,.825);
\draw[cyan, ultra thick] (.33,.33) -- (.33,.825);
\draw[cyan, ultra thick] (0,.66) -- (.495,.66);
\draw[cyan, ultra thick] (0,.495) -- (.495,.495);
\draw (0,0) -- (1,0) -- (1,1) -- (0,1) -- cycle;
\draw (.33,0) -- (.33, 1);
\draw (.66,0) -- (.66,1);
\draw (0,.66) -- (1,.66);
\draw (0,.33) -- (1,.33);
\draw (0,.495) -- (.66,.495);
\draw (0,.825) -- (.66,.825);
\draw (.495,.33) -- (.495,1);
\draw (.165,.33) -- (.165,1);
\draw (.165,0) -- (.165,.33);
\draw (.495,0) -- (.495,.33);
\end{tikzpicture},
\begin{tikzpicture}[scale=1.25]
\filldraw[draw=cyan, ultra thick, fill=cyan!50] (.165,.33) -- (.66,.33) -- (.66,.825) -- (.165,.825) -- cycle;
\draw[cyan, ultra thick] (.495,.33) -- (.495,.825);
\draw[cyan, ultra thick] (.33,.33) -- (.33,.825);
\draw[cyan, ultra thick] (.165,.66) -- (.66,.66);
\draw[cyan, ultra thick] (.165,.495) -- (.66,.495);
\draw (0,0) -- (1,0) -- (1,1) -- (0,1) -- cycle;
\draw (.33,0) -- (.33, 1);
\draw (.66,0) -- (.66,1);
\draw (0,.66) -- (1,.66);
\draw (0,.33) -- (1,.33);
\draw (0,.495) -- (.66,.495);
\draw (0,.825) -- (.66,.825);
\draw (.495,.33) -- (.495,1);
\draw (.165,.33) -- (.165,1);
\draw (.165,0) -- (.165,.33);
\draw (.495,0) -- (.495,.33);
\end{tikzpicture},\\
\begin{tikzpicture}[scale=1.25]
\filldraw[draw=cyan, ultra thick, fill=cyan!50] (0,0) -- (.495,0) -- (.495,.66) -- (0,.66) -- cycle;
\draw[cyan, ultra thick] (.165,0) -- (.165,.66);
\draw[cyan, ultra thick] (.33,0) -- (.33,.66);
\draw[cyan, ultra thick] (0,.33) -- (.495,.33);
\draw[cyan, ultra thick] (0,.495) -- (.495,.495);
\draw (0,0) -- (1,0) -- (1,1) -- (0,1) -- cycle;
\draw (.33,0) -- (.33, 1);
\draw (.66,0) -- (.66,1);
\draw (0,.66) -- (1,.66);
\draw (0,.33) -- (1,.33);
\draw (0,.495) -- (.66,.495);
\draw (0,.825) -- (.66,.825);
\draw (.495,.33) -- (.495,1);
\draw (.165,.33) -- (.165,1);
\draw (.165,0) -- (.165,.33);
\draw (.495,0) -- (.495,.33);
\end{tikzpicture},
\begin{tikzpicture}[scale=1.25]
\filldraw[draw=cyan, ultra thick, fill=cyan!50] (.165,0) -- (.66,0) -- (.66,.66) -- (.165,.66) -- cycle;
\draw[cyan, ultra thick] (.495,0) -- (.495,.66);
\draw[cyan, ultra thick] (.33,0) -- (.33,.66);
\draw[cyan, ultra thick] (.165,.33) -- (.66,.33);
\draw[cyan, ultra thick] (.165,.495) -- (.66,.495);
\draw (0,0) -- (1,0) -- (1,1) -- (0,1) -- cycle;
\draw (.33,0) -- (.33, 1);
\draw (.66,0) -- (.66,1);
\draw (0,.66) -- (1,.66);
\draw (0,.33) -- (1,.33);
\draw (0,.495) -- (.66,.495);
\draw (0,.825) -- (.66,.825);
\draw (.495,.33) -- (.495,1);
\draw (.165,.33) -- (.165,1);
\draw (.165,0) -- (.165,.33);
\draw (.495,0) -- (.495,.33);
\end{tikzpicture},
\begin{tikzpicture}[scale=1.25]
\filldraw[draw=cyan, ultra thick, fill=cyan!50] (.33,0) -- (1,0) -- (1,1) -- (.33,1) -- cycle;
\draw[cyan, ultra thick] (.495,0) -- (.495,1);
\draw[cyan, ultra thick] (.66,0) -- (.66,1);
\draw[cyan, ultra thick] (.33,.66) -- (1,.66);
\draw[cyan, ultra thick] (.33,.33) -- (1,.33);
\draw (0,0) -- (1,0) -- (1,1) -- (0,1) -- cycle;
\draw (.33,0) -- (.33, 1);
\draw (.66,0) -- (.66,1);
\draw (0,.66) -- (1,.66);
\draw (0,.33) -- (1,.33);
\draw (0,.495) -- (.66,.495);
\draw (0,.825) -- (.66,.825);
\draw (.495,.33) -- (.495,1);
\draw (.165,.33) -- (.165,1);
\draw (.165,0) -- (.165,.33);
\draw (.495,0) -- (.495,.33);
\end{tikzpicture}.
\end{tabular}
\end{minipage}\hspace{1.5cm}
}
\caption{Example of a vertical tensor expansion. We consider five LR B-splines of bidegree $(2,2)$, namely $B$ and $B^1, \ldots, B^4$, with $\supp B^1,\ldots, \supp B^4$ contained in the upper left corner of $\supp B$. The mesh $\cN$, of multiplicity 1, generated by the meshlines of $B, B^1, \ldots, B^4$ is depicted in (a) and the supports of the LR B-splines are shown in (b). In (c) we perform a  vertical tensor expansion of $B^1, \ldots, B^4$ in $B$. In (d) the supports of the new set of LR B-splines are shown: none of them has a nested LR B-spline anymore.
}\label{exTensorEx}
\end{figure}

The \NtwoS-structured mesh refinement is defined algorithmically as follows. We start from a structured mesh refinement to obtain a new set of LR B-splines. We then collect in a set $\mathcal{B}$ all those LR B-splines that have nested LR B-splines in their supports. If $\mathcal{B}$ is non-empty, we select an LR B-spline $B$ in $\mathcal{B}$ and we apply a \unilateral{} tensor expansion to it. This triggers a refinement of the LR B-spline set, and therefore it changes also the set $\mathcal{B}$. We repeat this procedure till $\mathcal{B}$ becomes empty. In Theorem~\ref{N2Sproperty} we shall prove that this always happens in a finite number of steps.
% Thus, when finally $\mathcal{B} = \emptyset$, we can apply again the structured mesh refinement to the LR B-splines contributing more to the error and start over the process.
This procedure is sketched in Algorithm~\ref{alg:N2S-mesh}.
The \unilateral{} tensor expansions are performed by alternating the direction for $i$ even and odd, respectively, in order to bound the thinning of the box-partition elements in a specific direction and preserve the uniformity of the mesh as much as possible.
The LR-mesh obtained in this way will be called an \highlight{\NtwoS-structured LR-mesh}, or in short \highlight{\NtwoStwo{} LR-mesh}.

\begin{algorithm}[b]
\setlength{\algoheightrule}{0pt}
\setlength{\algotitleheightrule}{0pt}
$\mathcal{B}_1$ is the B-spline set on the open tensor mesh equal to the domain's boundary\;
\For{$i=1,2,\ldots$}{
perform a structured mesh refinement of $\mathcal{B}_i$\;
initialize $\mathcal{B}_{i+1}$ as the LR B-spline set defined on the new LR-mesh\; \label{alg:step3}
define $\mathcal{B} = \{B \in \mathcal{B}_{i+1} : \exists\,B' \in \mathcal{B}_{i+1}\mbox{ with }B' \preceq B\}$\;\label{alg:step4}
\While{$\mathcal{B} \neq \emptyset$}{
select $B \in \mathcal{B}$\;
perform a \unilateral{} tensor expansion of the LR B-splines nested in $B$\;
update $\mathcal{B}_{i+1}$ as the LR B-spline set defined on the new LR-mesh\;
update $\mathcal{B} = \{B \in \mathcal{B}_{i+1} : \exists\,B' \in \mathcal{B}_{i+1}\mbox{ with }B' \preceq B\}$\;
}}
\caption{\NtwoS-structured mesh refinement.}\label{alg:N2S-mesh}
\end{algorithm}
\begin{thm}\label{N2Sproperty}
Given an axis-aligned rectangular domain $\Omega \subseteq \mathbb{R}^2$, let $\mathcal{B}_1$ be the set of standard bivariate B-splines defined on the open tensor mesh whose meshlines are the edges of $\partial \Omega$. Then,
\begin{enumerate}
\item the LR B-spline sets $\mathcal{B}_i$ provided by Algorithm~\ref{alg:N2S-mesh} are well defined, i.e., the set $\mathcal{B}$ of the algorithm becomes empty in a finite number of iterations, for every index $i \geq 2$;
\item all the LR B-splines in $\mathcal{B}_i$ are non-nested, for every $i \geq 1$.
\end{enumerate}
\end{thm}
\begin{proof}
Without loss of generality, we can assume that $\Omega = [0,1]\times [0,1]$. We proceed by induction on the index of the B-spline set. For $i=1$, $\mathcal{B}_1$ is the set of standard B-splines on the open tensor mesh equal to the domain's boundary and we know they are locally linearly independent. By Theorem~\ref{TeoBressan} this is equivalent to be all non-nested. Assume now that $\mathcal{B}_i$ for some $i\geq1$ is well defined and that the functions in it are all non-nested. Let us then prove that also $\mathcal{B}_{i+1}$ is well defined and there is no LR B-spline nested into another LR B-spline of it. At every loop iteration in the algorithm, the LR B-splines that have a nested LR B-spline in their support are collected in the set $\mathcal{B}$. Therefore, whenever we can show that $\mathcal{B}$ becomes empty after a certain iteration of the loop, we can immediately conclude both statements in the theorem.

By Corollary~\ref{lemmaTE}, all the \unilateral{} tensor expansions performed to define the set $\mathcal{B}_{i+1}$ can be done in the same direction $k \in \{1,2\}$, which is therefore fixed once and for all by the index $i+1$.
The length of the LR B-spline supports in the $(3-k)$-th direction at any iteration of the loop cannot become shorter than $2^{-(i+1)}$ regardless of the number of \unilateral{} tensor expansions applied until then. This is because the $(3-k)$-splits on the LR-meshes defined in the loop are fixed by the structured mesh refinement performed on $\mathcal{B}_i$ at the beginning of the process and the minimal length of the box-partition elements in the $(3-k)$-th direction is $2^{-(i+1)}$. Therefore, the split extensions applied when performing a \unilateral{} tensor expansion in the $k$-th direction have lengths bounded from below by $2^{-(i+1)}$ in all the steps of the loop. This means that in a finite number of \unilateral{} tensor expansions a $k$-split could be extended up to the domain's boundary, if needed, to remove nestedness issues, as these extensions cannot become arbitrarily small.
In the worst case scenario, we must extend all the $k$-splits to cross entirely the domain. However, in this case, the resulting LR-mesh would be tensorized in the $k$-th direction. By Proposition~\ref{N2Stensorized}, there are only non-nested LR B-splines on this LR-mesh and thus $\mathcal{B}$ becomes empty in a finite number of loop iterations.
\end{proof}

In practice, the loop related to $\mathcal{B}$ stops quickly and the \NtwoStwo{} LR-meshes are far from being entirely tensorized in one direction. In Figure~\ref{diag} we depict (a) the structured LR-mesh, (b) the corresponding \NtwoStwo{} LR-mesh, and (c) the LR-mesh proposed in \cite{meshbressan}, obtained by performing 7 iterations of diagonal refinement in $[0,1]^2$, using bidegree $(2,2)$. For ease of comparison, we also indicate the number of LR B-splines defined on each of these meshes. We recall that the LR B-splines are not locally linearly independent on the structured LR-mesh, whereas they are on the \NtwoStwo{} LR-mesh and the LR-mesh proposed in \cite{meshbressan}.
\begin{figure}[t!]
\centering
\subfloat[1430 LR B-splines]{
\includegraphics[width=5.25cm, trim={12.5cm 2cm 11.5cm 1cm},clip]{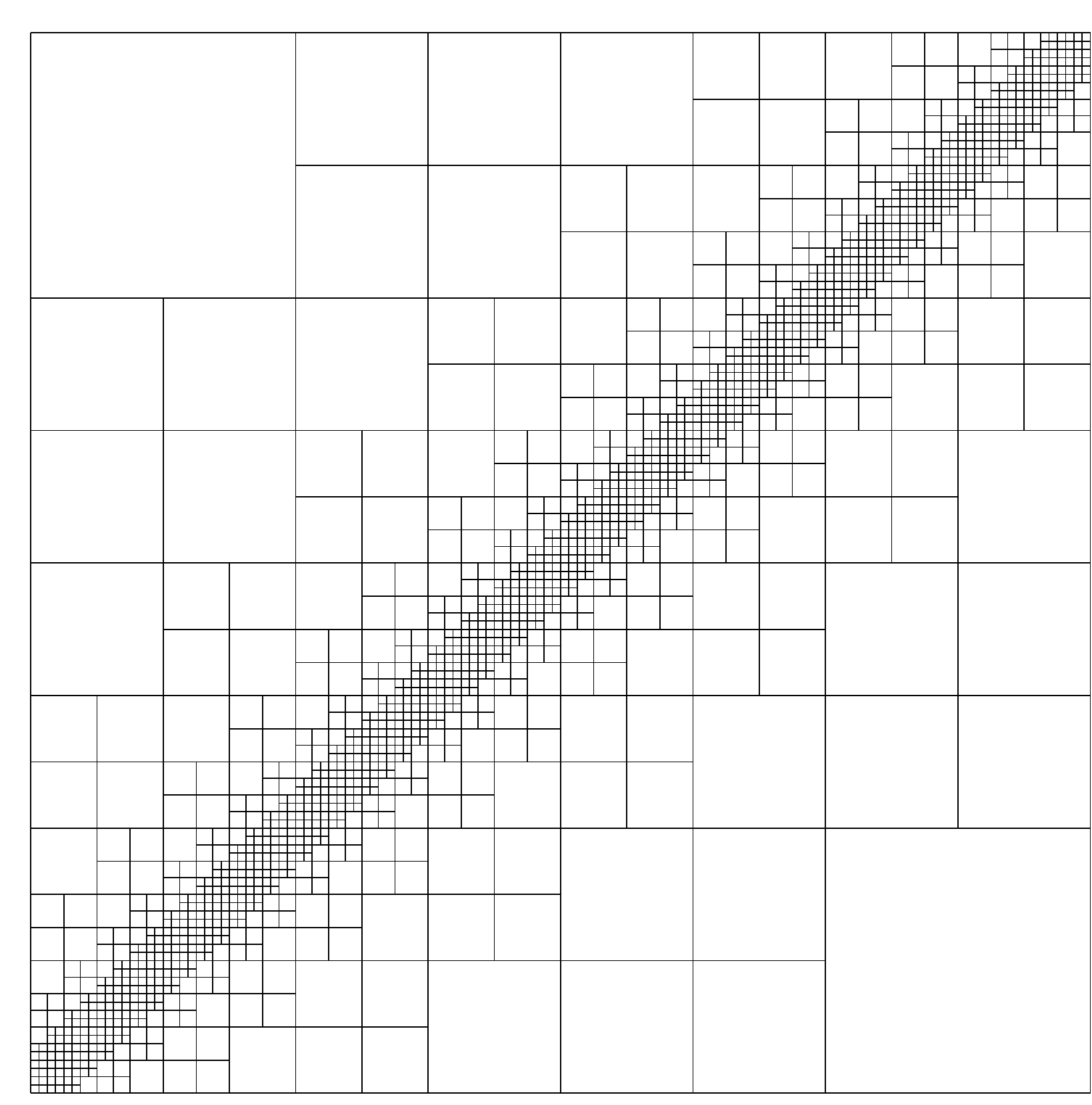}
}
\subfloat[1894 LR B-splines]{
\includegraphics[width=5.25cm, trim={12.5cm 2cm 11.5cm 1cm},clip]{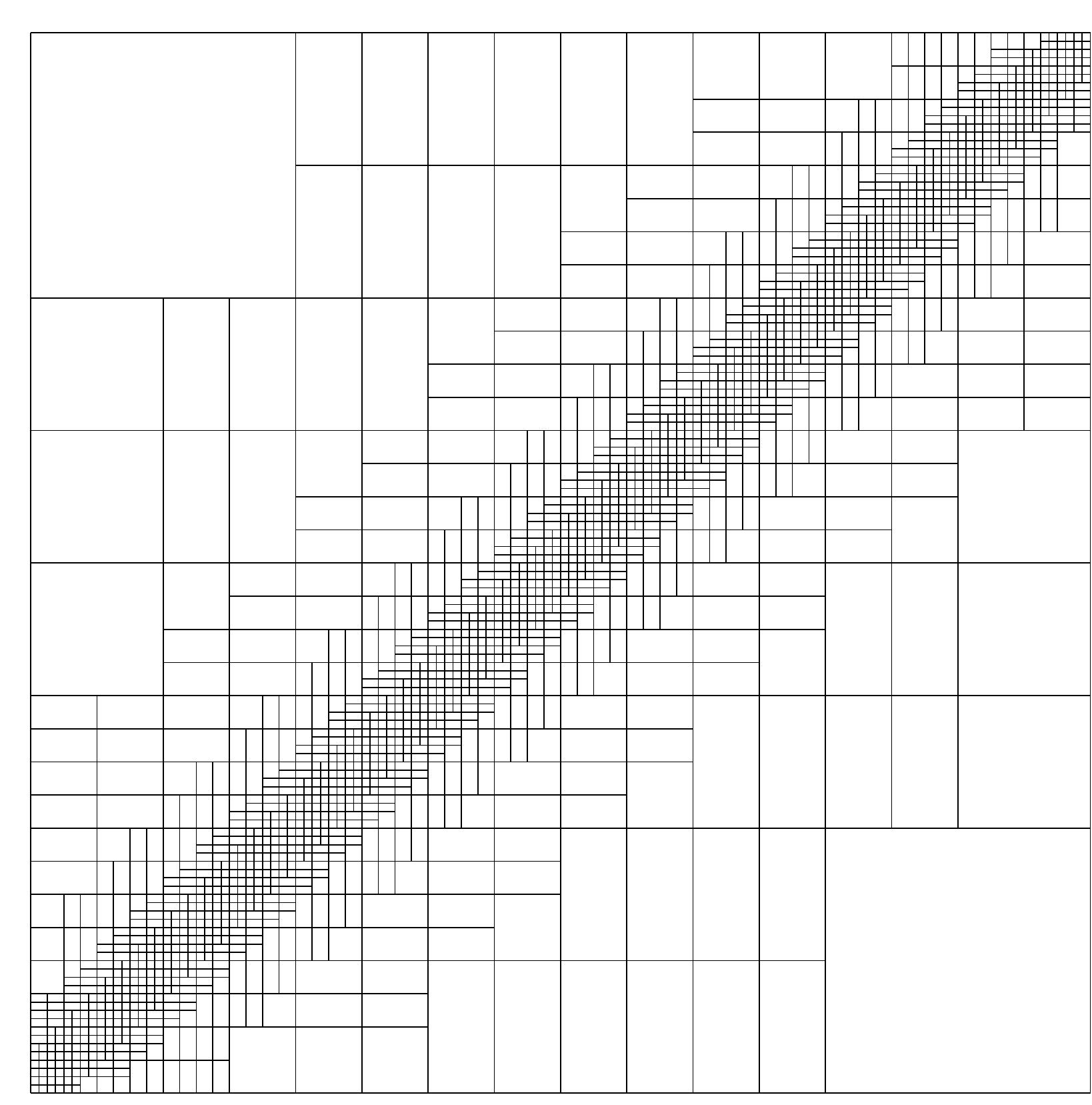}
}
\subfloat[2243 LR B-splines]{
\includegraphics[width=5.25cm, trim={12.5cm 2cm 11.5cm 1cm},clip]{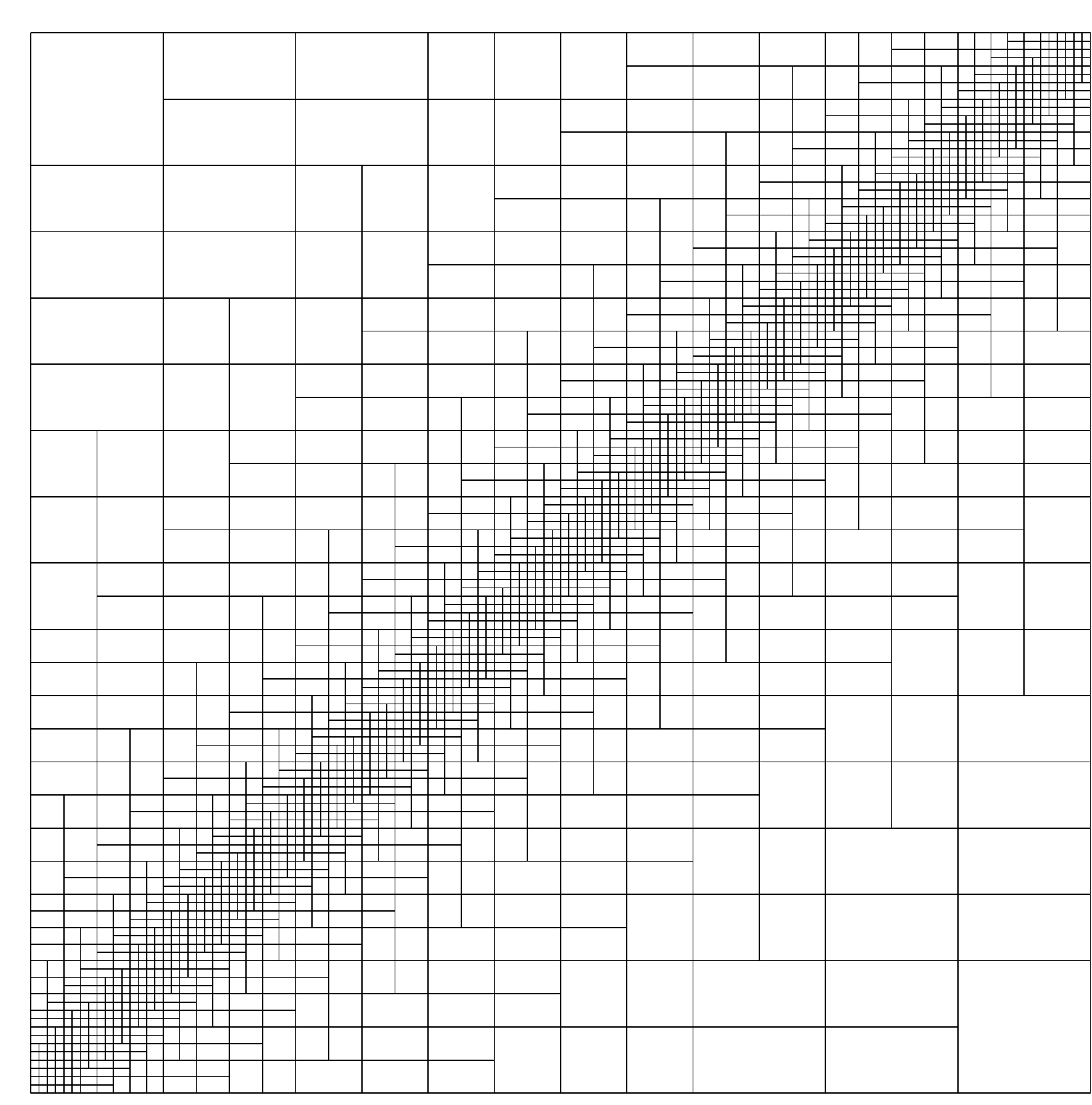}
}
\caption{Visual comparison between meshes of bidegree $(2,2)$ obtained by performing 7 iterations of different mesh refinement strategies along the diagonal: (a) the structured LR-mesh, (b) the \NtwoStwo{} LR-mesh, and (c) the LR-mesh proposed in \cite{meshbressan}.}\label{diag}
\end{figure}

In Figure~\ref{diag}(b) and Figure~\ref{N2SEx} one can see how the refinement in the \NtwoStwo{} LR-meshes propagates from the region where the structured mesh refinement has been applied. In all the considered cases, the refinement does not heavily spread out. It is important to highlight, however, that the prolongation of the splits needed to recover the \NtwoS-property is not unique. Indeed, when refining an LR B-spline to remove nestedness issues, the inserted split prolongations refine not only the considered LR B-spline but in general also other LR B-splines in the neighborhood. Then, some of the newly introduced neighboring LR B-splines might not need a \unilateral{} tensor expansion anymore and the ordering used for removing nestedness has thus an effect on the resulting mesh.

One might consider to treat all the LR B-splines ``in parallel'', i.e., first collect all the split extensions needed to remove nestedness in all the LR B-splines requiring a treatment and then insert all of them at the same time to refine the function basis. This could result in a more uniform propagation of the refinement out of the region where the structured mesh has been applied. %%And still the mesh is not good looking I think...In order to get something gradual and smooth like bressan we have to refine also B-splines that do not have nested B-splines.
On the other hand, by doing this, some split extensions could be unnecessary for recovering the \NtwoS-property. Therefore, in general, also the number of LR B-splines on these meshes would be higher than the number obtained when treating one LR B-spline at a time. In the examples presented in this paper we do not remove nestedness ``in parallel''. Hence, the resulting \NtwoStwo{} LR-meshes depend on the order used when the \unilateral{} tensor expansions are applied. On the other hand, the number of LR B-splines will be closer to the number of LR B-splines obtained when performing only the structured mesh refinement, i.e., closer to the ``optimal'' number of LR B-splines needed to reduce the error while preserving the local linear independence.
\begin{figure}[t!]
\centering
\subfloat[10281 LR B-splines]{
\includegraphics[width=6cm, angle=270, origin=c, trim={12.5cm 2cm 11.5cm 1cm},clip]{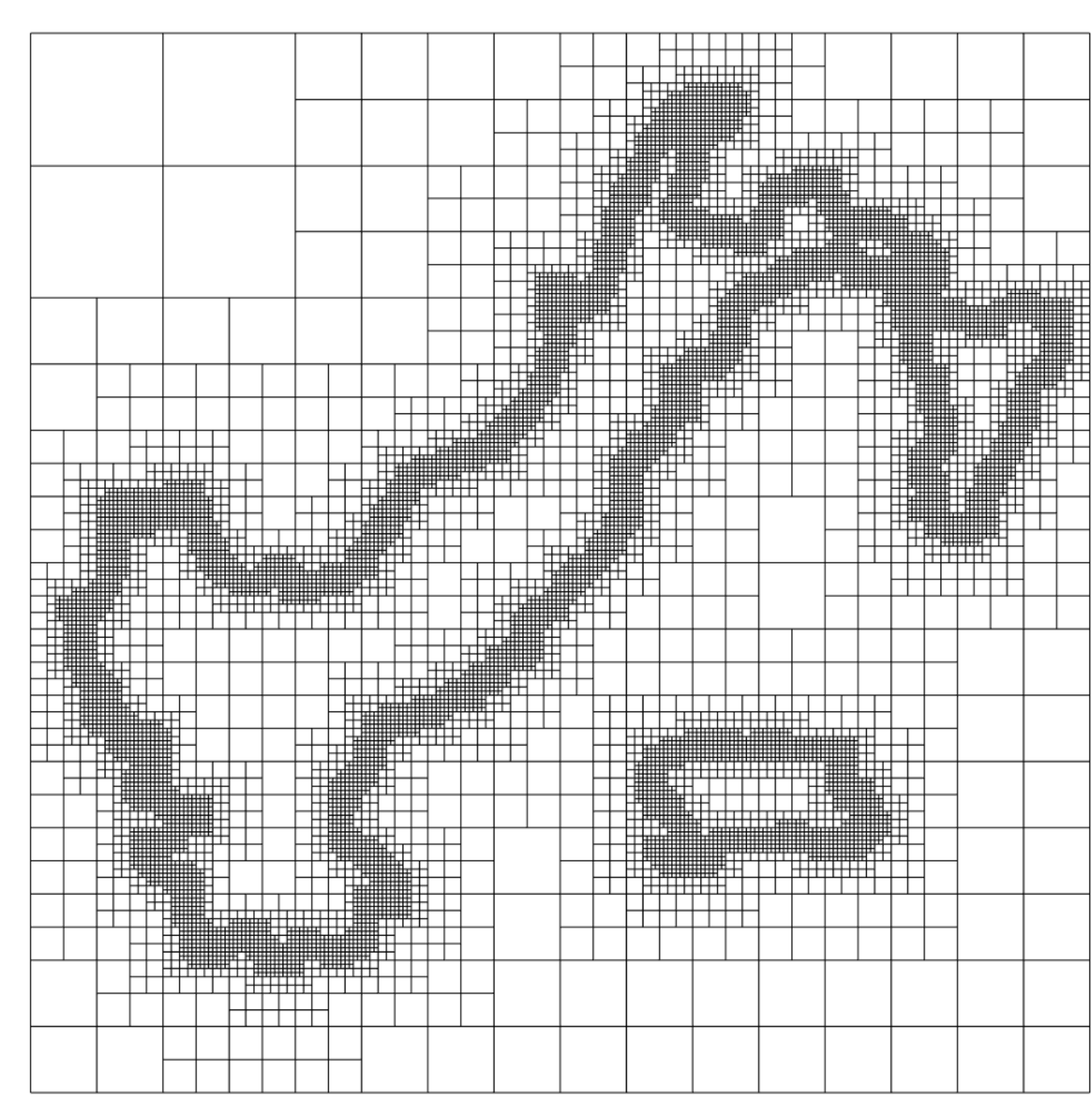}
}\quad
\subfloat[12438 LR B-splines]{
\includegraphics[width=6cm, angle=270, origin=c, trim={12.5cm 2cm 11.5cm 1cm},clip]{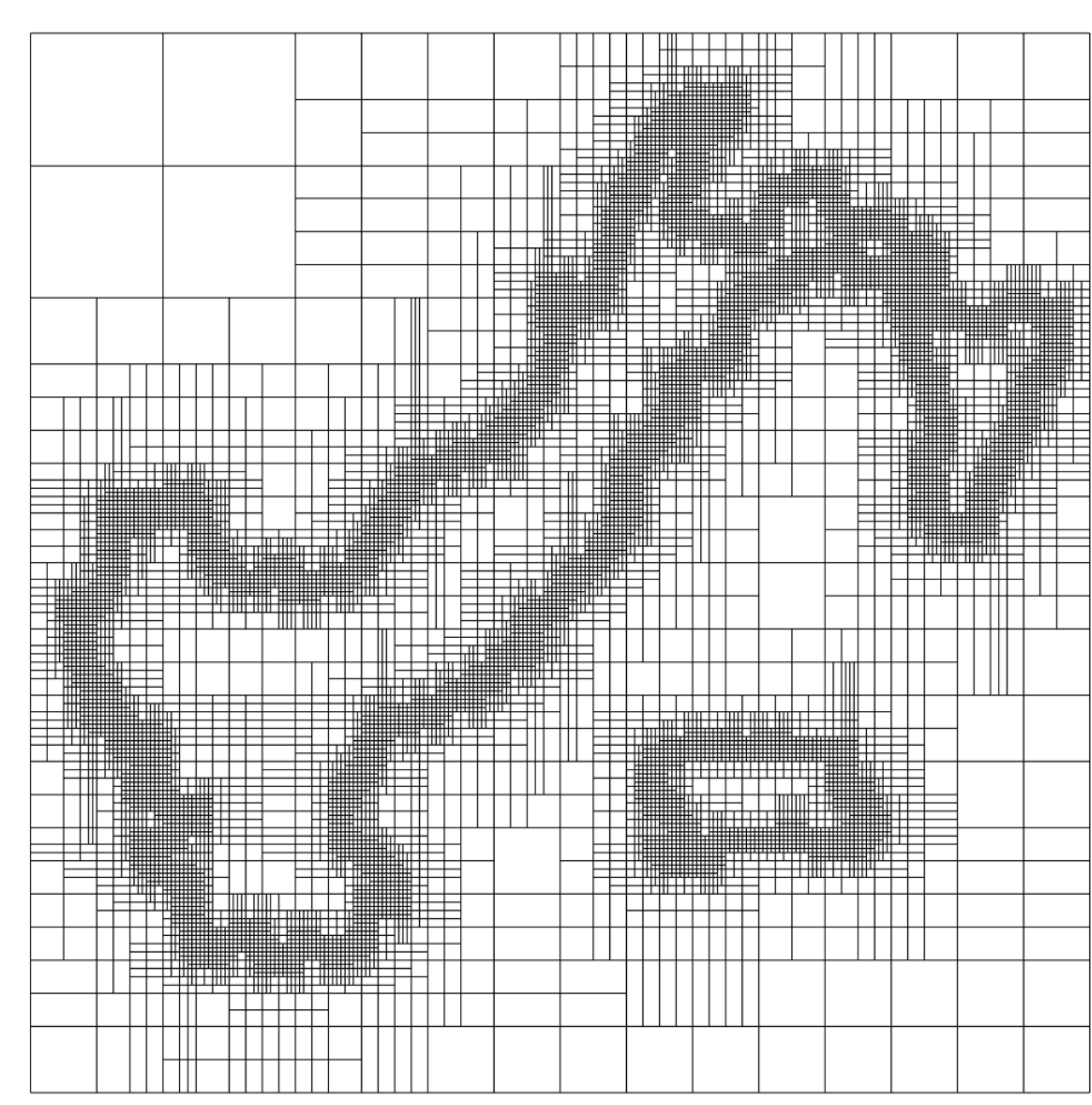}}\\
\subfloat[13459 LR B-splines]{
\includegraphics[width=6cm, angle=270, origin=c, trim={12.5cm 2cm 11.5cm 1cm},clip]{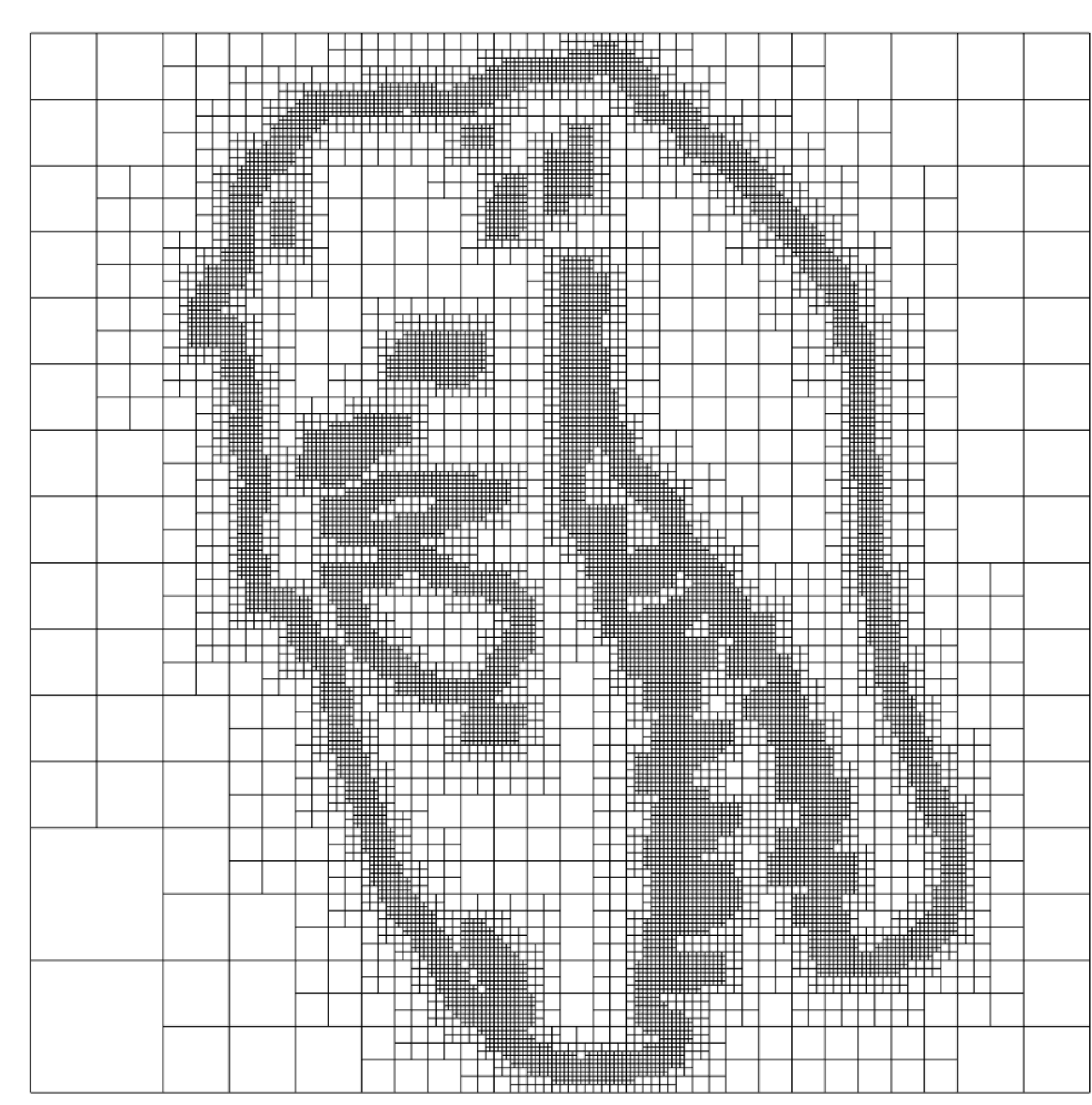}}\quad
\subfloat[15993 LR B-splines]{
\includegraphics[width=6cm, angle=270, origin=c, trim={12.5cm 2cm 11.5cm 1cm},clip]{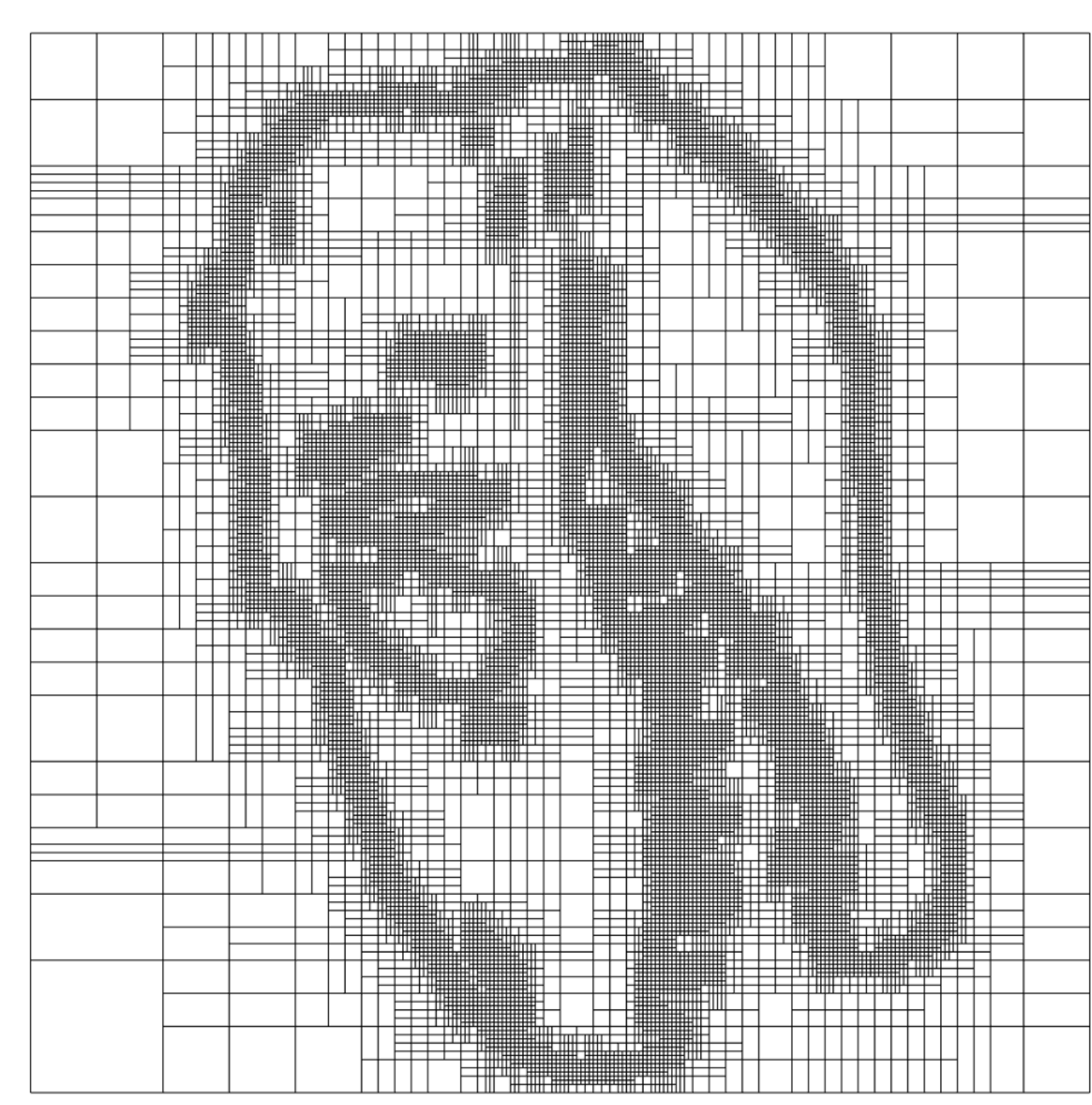}}\\
\subfloat[8608 LR B-splines]{
\includegraphics[width=6cm, angle=270, origin=c, trim={12.5cm 2cm 11.5cm 1cm},clip]{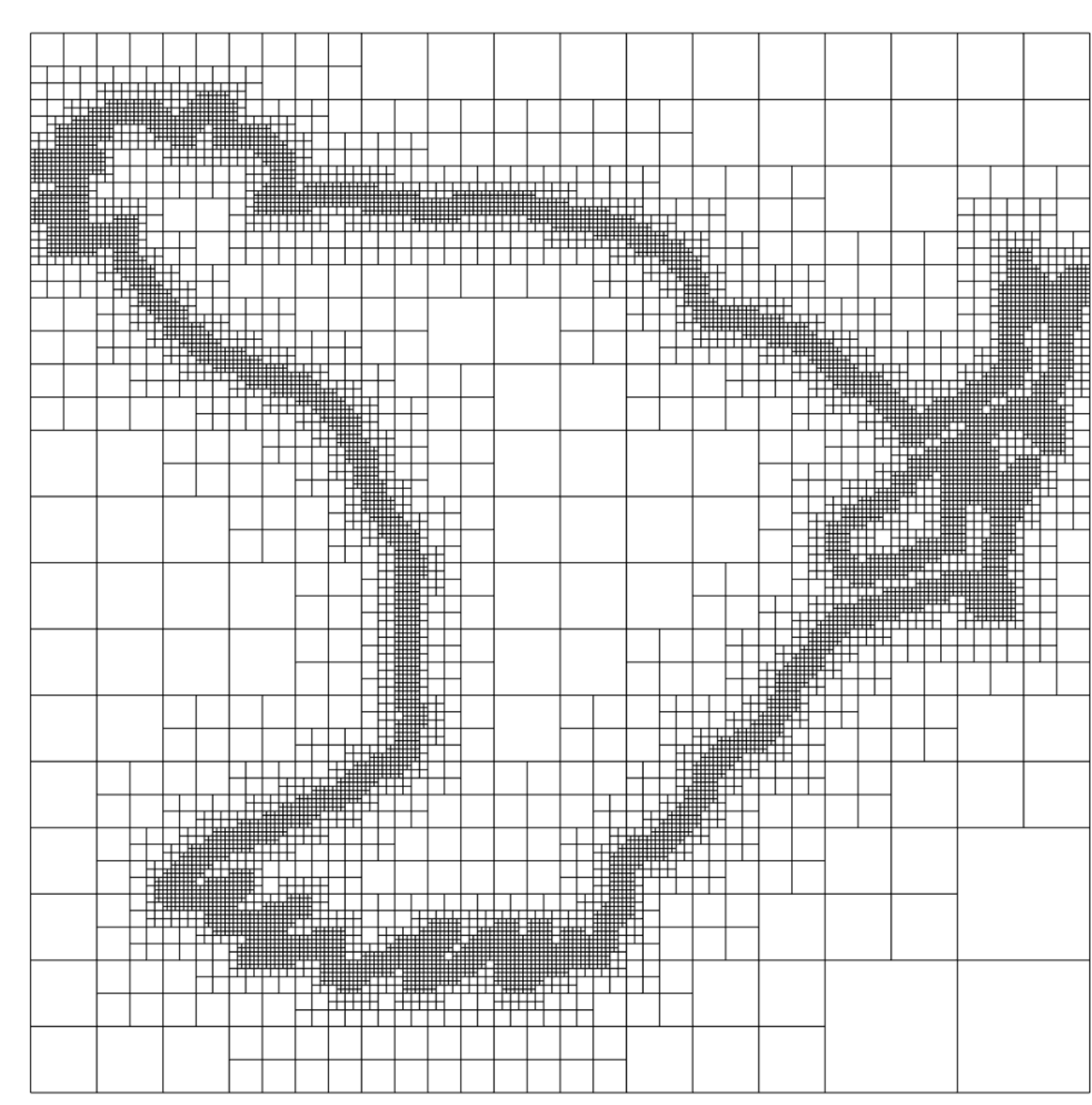}}\quad
\subfloat[10841 LR B-splines]{
\includegraphics[width=6cm, angle=270, origin=c, trim={12.5cm 2cm 11.5cm 1cm},clip]{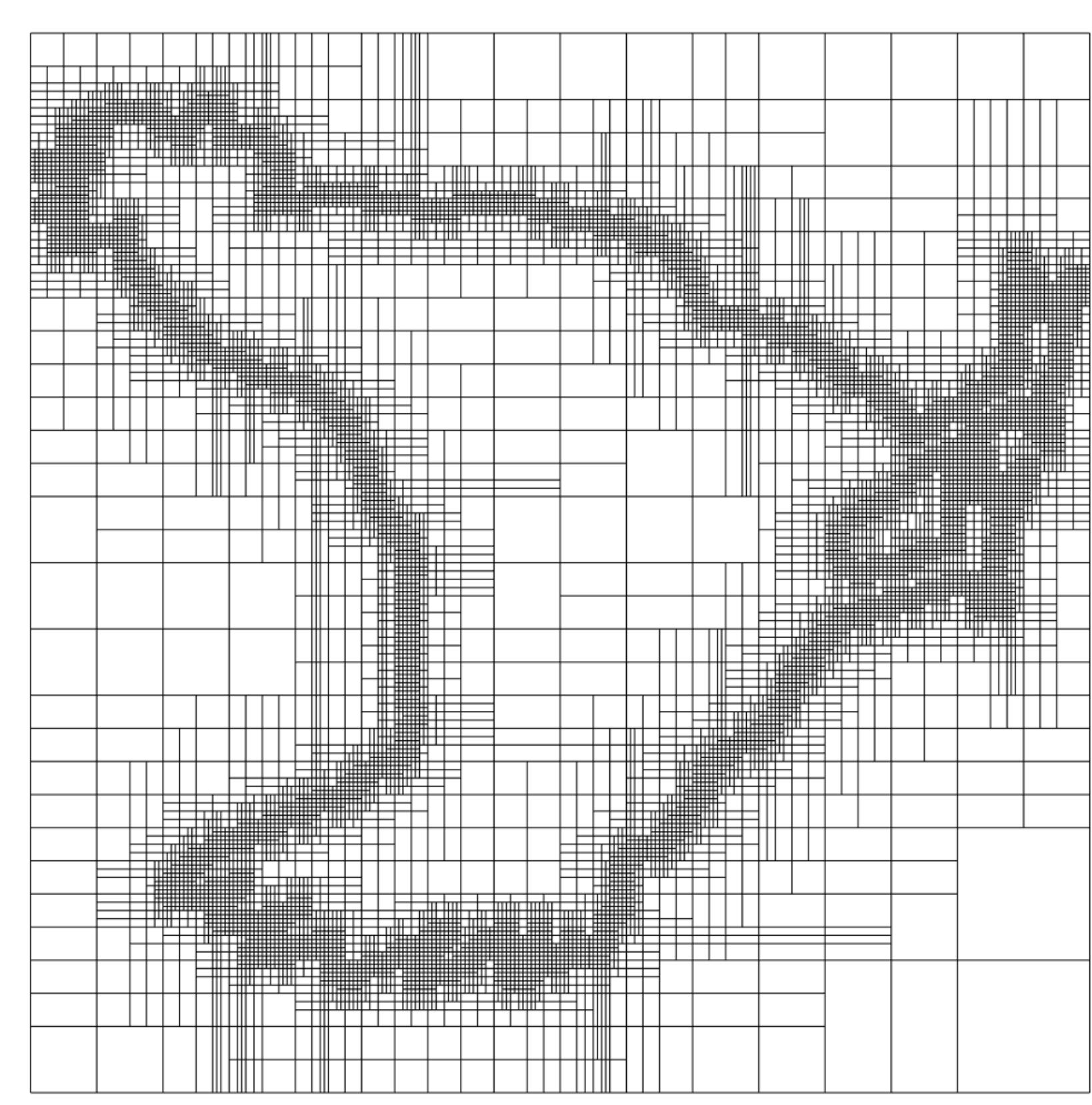}}
\caption{Meshes of bidegree $(2,2)$ obtained by performing 8 iterations of mesh refinement on 3 different regions resulting in a structured LR-mesh (left column) and an \NtwoStwo{} LR-mesh (right column).}\label{N2SEx}
\end{figure}

We finally remark that one might also opt for full tensor expansions in the supports, instead of \unilateral{} tensor expansions, to solve nestedness issues. The proof of Theorem~\ref{N2Sproperty} could be easily rephrased for the case of full tensor expansions. The key is that we only prolong splits provided by the structured mesh refinement performed at the beginning of the process. Therefore, if we do full tensor expansions, in the worst case scenario we would end up with a standard tensor mesh of size $h= 2^{-(i+1)}$ to define the set $\mathcal{B}^{i+1}$, instead of an LR-mesh tensorized in one direction. In such case, $\mathcal{B}$  would still become empty in a finite number of loop iterations. However, we decided to do the expansion of the splits only in one direction at a time because
%it reduces the propagation more.
it results in the least propagation.

%------------------------------------------------------------------------

\section{Application I: Quasi-interpolation}\label{sec:quasi-interpolation}
A quasi-interpolation method is a procedure to compute the coefficients assigned to the basis elements of a prescribed function space, with the aim of approximating a given arbitrary function or data set.
The resulting approximant is called a \highlight{quasi-interpolant (QI)}. The computation of any of such coefficients may depend only on the data/function restricted to the corresponding basis element's support (local method), and perhaps some neighboring other basis elements' supports, or it can depend on the data/function in the entire domain (global method), as in the least-squares method.
Given a function $f$ and an approximation space, whose basis is denoted by $\mathcal{B}$, we write a related QI in the form
\begin{equation*}%\label{QI}
\mathfrak{Q}f := \sum_{B \in \mathcal{B}} \lambda_B(f) B,
\end{equation*}
where $\lambda_B(f)$ is the coefficient of the basis element $B\in \mathcal{B}$ computed by the selected method.
\begin{definition}
A quasi-interpolation method such that $\mathfrak{Q}f = f$ for all $f$ in a space $V$ is said to \highlight{reproduce} the space $V$.
\end{definition}

When using spline spaces of bidegree $\pmb{p}$ as approximation spaces, a common requirement is that the polynomial space $\Pi_{\pmb{p}}$ is reproduced by the quasi-interpolation method, in order to ensure good approximation properties.
A general recipe for constructing local quasi-interpolation methods for tensor spline spaces, with the polynomial reproduction property, can be found in \cite{QIlyche}.
\begin{rec}\label{rec:QI-TP}
Let $f$ be a given function defined on the rectangle $\Omega$. Given a bidegree $\pmb{p}$, let
$\widetilde{\cM}$ be an open tensor mesh on $\Omega$,
and let $\mathcal{B}(\widetilde{\cM})$ be the set of tensor B-splines of bidegree $\pmb{p}$ on $\widetilde{\cM}$.
We compute the coefficient $\lambda_B(f)$, for every $B = B[\pmb{x}, \pmb{y}] \in \mathcal{B}(\widetilde{\cM})$, as follows.
\begin{enumerate}
\item Let $U\subseteq \RR^2$ be an open set that intersects the interior of $\supp B$ (for instance, $U$ can be a box-partition element of $\cM(\pmb{x}, \pmb{y})$), and let $\mathcal{B}(U)$ be the subset of $\mathcal{B}(\widetilde{\cM})$ consisting of all the tensor B-splines not identically zero on $U$.
\item Choose a local polynomial approximation method $\mathfrak{P}_U$ such that $\mathfrak{P}_Ug = g$ for all $g \in \Pi_{\pmb{p}}$ defined on $U$ (typical choices are least-squares or interpolation methods).
Letting $g_{|U}$ be the restriction of $g$ to $U$, we can write
$$(\mathfrak{P}_Uf)_{|U} = \sum_{\widetilde{B} \in \mathcal{B}(U)} b_{\widetilde{B}}(f) \widetilde{B}_{|U},$$
for some coefficients $b_{\widetilde{B}}(f)$ provided by the chosen local approximation method.
\item Since $B \in \mathcal{B}(U)$, set $\lambda_B(f) := b_{B}(f)$.
\end{enumerate}
Then, define
\begin{equation*}%\label{QI-TP-proc}
\mathfrak{Q}f := \sum_{B \in \mathcal{B}(\widetilde{\cM})} \lambda_B(f) B.
\end{equation*}
\end{rec}

Inspired by the above recipe for tensor splines and similar to the local quasi-interpolation strategy developed for THB-splines in \cite{effortless,effortless-error}, we can formulate a general recipe for constructing QIs in the space spanned by $\mathcal{B}^{\mathcal{LR}}(\cM)$ on a given open LR-mesh $\cM$ as follows:
select for each LR B-spline $B$ in $\mathcal{B}^{\mathcal{LR}}(\cM)$ a local tensor space containing $B$, and pick the coefficient corresponding to $B$ in the expression of any QI in such a tensor space. In particular, when the smallest local tensor space  containing each basis function $B$ is considered, we arrive at the following recipe.
\begin{rec}\label{rec:QI-LR}
Let $f$ be a given function defined on the rectangle $\Omega$.
Given a bidegree $\pmb{p}$, let $\cM$ be an open LR-mesh on $\Omega$, and let $\mathcal{B}^{\mathcal{LR}}(\cM)$ be the set of LR B-splines of bidegree $\pmb{p}$ on $\cM$.
We compute the coefficient $\lambda_B(f)$, for every $B = B[\pmb{x}, \pmb{y}]\in \mathcal{B}^{\mathcal{LR}}(\cM)$, as follows.
\begin{enumerate}
\item Let $\widetilde{\cM}_B$ be the open (tensor) mesh obtained by rising the boundary meshline multiplicities of $\cM_B =\cM(\pmb{x}, \pmb{y})$ to full multiplicity, and
let $\mathcal{B}(\widetilde{\cM}_B)$ be the set of tensor B-splines defined on $\widetilde{\cM}_B$.
\item Consider a quasi-interpolation method in the space spanned by $\mathcal{B}(\widetilde{\cM}_B)$,
    $$\mathfrak{Q}_Bf = \sum_{\widetilde{B} \in \mathcal{B}(\widetilde{\cM}_B)} b_{\widetilde{B}}(f) \widetilde{B},$$
    reproducing all $g \in \Pi_{\pmb{p}}$ (for instance, use Recipe~\ref{rec:QI-TP}).
\item Since $B \in \mathcal{B}(\widetilde{\cM}_B)$, set $\lambda_B(f) := b_{B}(f)$.
\end{enumerate}
Then, define
\begin{equation*}
%\label{QI-LR-proc}
\mathfrak{Q}f := \sum_{B \in \mathcal{B}^{\mathcal{LR}}(\cM)} \lambda_B(f) B.
\end{equation*}
\end{rec}
Since $B\in \mathcal{B}(\widetilde{\cM}_B)$ for any $B\in \mathcal{B}^{\mathcal{LR}}(\cM)$,
the function $\mathfrak{Q}f$ in Recipe~\ref{rec:QI-LR} is well defined. Moreover, it will reproduce polynomials on the entire domain if the LR-mesh has the \NtwoS-property as stated in the following proposition.

\begin{prop}\label{polyre}
Given a bidegree $\pmb{p}$, let $\cM$ be an open LR-mesh and let $\mathcal{B}^{\mathcal{LR}}(\cM)$ be the set of LR B-splines of bidegree $\pmb{p}$ on $\cM$. Assume that $\cM$ has the \NtwoS-property, then
$$\mathfrak{Q}g = g, \quad \forall g \in \Pi_{\pmb{p}},$$
where the quasi-interpolation operator $\mathfrak{Q}$ is defined in Recipe~\ref{rec:QI-LR}.
\end{prop}%
\begin{proof}
From \cite[Theorem~4.6]{someproperties}, if $\cM$ has the \NtwoS-property, then for all $g \in \Pi_{\pmb{p}}$
we have
$$
g=\sum_{B \in \mathcal{B}^{\mathcal{LR}}(\cM)} g_B B,\quad g_B\in \RR,
$$
where for all $B \in \mathcal{B}^{\mathcal{LR}}(\cM)$, the coefficient $g_B$ only depends on $g$ and on the knots defining the LR B-spline $B$. Therefore, $g_B$ remains the same if we represent $g$ in any set of tensor B-splines  containing $B$. Since, according to Recipe~\ref{rec:QI-LR}, any $\mathfrak{Q}_B$ reproduces all polynomials in $ \Pi_{\pmb{p}}$ we have
$$
g_B=\lambda_B(g), \quad \forall B  \in \mathcal{B}^{\mathcal{LR}}(\cM), \quad g\in \Pi_{\pmb{p}},
$$
which completes the proof.
\end{proof}
\begin{figure}[h!]\hspace*{-1.0cm}
%\centering
\begin{floatrow}
\ffigbox{%
\includegraphics[width=8cm, trim={1cm 2.15cm 0cm 5.5cm},clip]{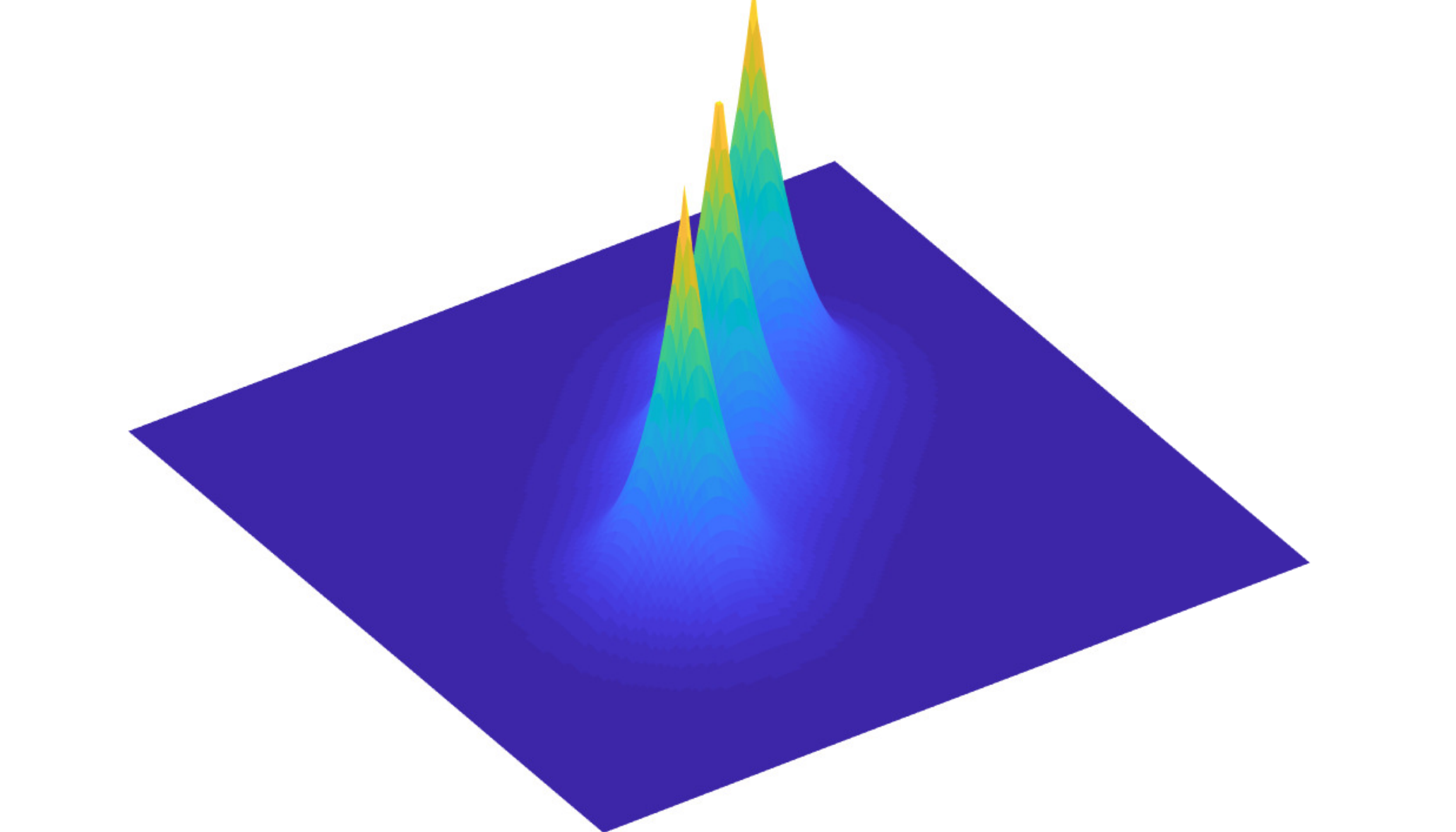}
}{%
  \caption{Transcendent function in $[-1,1]^2$.}\label{figex7}%
}
\capbtabbox{%
\footnotesize{
\begin{tabular}{l|ll|l}
\cline{2-3}
                             & \multicolumn{2}{c|}{\# basis functions} &                                \\ \hline
\multicolumn{1}{|l|}{levels} & Tensor mesh       & \NtwoStwo{} LR-mesh      & \multicolumn{1}{l|}{max error} \\ \hline
\multicolumn{1}{|l|}{1}      & 36                & 36                  & \multicolumn{1}{l|}{5.686e-01} \\ \hline
\multicolumn{1}{|l|}{2}      & 100               & 86                  & \multicolumn{1}{l|}{4.645e-01} \\ \hline
\multicolumn{1}{|l|}{3}      & 324               & 161                 & \multicolumn{1}{l|}{2.575e-01} \\ \hline
\multicolumn{1}{|l|}{4}      & 1156              & 254                 & \multicolumn{1}{l|}{1.472e-01} \\ \hline
\multicolumn{1}{|l|}{5}      & 4356              & 363                 & \multicolumn{1}{l|}{5.955e-02} \\ \hline
\multicolumn{1}{|l|}{6}      & 16900             & 450                 & \multicolumn{1}{l|}{2.156e-02} \\ \hline
\multicolumn{1}{|l|}{7}      & 66564             & 537                 & \multicolumn{1}{l|}{1.415e-02} \\ \hline
\end{tabular}}
}{%
  \caption{B-spline set cardinalities of bidegree $(2,2)$ for different levels of maximal resolution.}\label{ex7}%
}
\end{floatrow}
\end{figure}
We have tested the quasi-interpolation strategy described in Recipe~\ref{rec:QI-LR} on \NtwoStwo{} LR-meshes to approximate polynomials and transcendent functions. Given an \NtwoStwo{} LR-mesh $\cM$, this recipe requires the construction of a QI based on $\mathcal{B}(\widetilde{\cM}_B)$ for each of the LR B-splines $B \in \mathcal{B}^{\mathcal{LR}}(\cM)$ of bidegree $\pmb{p}= (p_1,p_2)$. Following Recipe~\ref{rec:QI-TP}, we have used interpolation as local approximation method in the computation of these QIs. More precisely, we have selected a unisolvent set of $(p_1+1)(p_2+1)$ interpolation points, organized in a tensor grid over a single box-partition element of $\widetilde{\cM}_B$, and then we have set a linear system by evaluating $f$ and the tensor B-splines in $\mathcal{B}(\widetilde{\cM}_B)$ at these points.
%By sampling such a number of points in the same box-partition element,
This guarantees polynomial reproduction (actually spline reproduction) of the quasi-interpolation method in the spaces $\mathcal{B}(\widetilde{\cM}_B)$, for $B$ varying in $\mathcal{B}^{\mathcal{LR}}(\cM)$. Therefore, also the resulting quasi-interpolation method on $\mathcal{B}^{\mathcal{LR}}(\cM)$ has the polynomial reproduction property thanks to Proposition~\ref{polyre}. Indeed, in all the tests with polynomial functions of bidegree at most $\pmb{p}$, the maximum error was in the order of the machine precision, regardless of the number of iterations performed to construct the \NtwoStwo{} LR-mesh. The maximum error was computed on a uniform 150 $\times$ 150 grid.

As test with a transcendent function, we have considered $$
f(x,y) = \frac{2}{3}e^{-\sqrt{(10x-3)^2+(10y-3)^2}}+\frac{2}{3}e^{-\sqrt{(10x+3)^2+(10y+3)^2}}+\frac{2}{3}e^{-\sqrt{(10x)^2+(10y)^2}},
$$
which is characterized by three steep peaks on the square $[-1,1]^2$ located at $(-0.3,-0.3)$, $(0,0)$ and $(0.3,0.3)$; see Figure~\ref{figex7}.
This function has also been used in \cite{effortless} to investigate the approximation power of a similar quasi-interpolation method developed for THB-splines.
In Table~\ref{ex7}, we compare the number of basis functions of bidegree $(2,2)$ when considering global tensor meshes and local \NtwoStwo{} LR-meshes for different levels of maximal resolution (for level $\ell$, the smallest box-partition elements on the mesh have length $2^{-\ell}$). The \NtwoStwo{} LR-mesh is obtained by refining the LR B-splines whose supports contain one of the three points where a peak occurs via structured mesh refinement and then by recovering the \NtwoS-property via \unilateral{} tensor expansions. For a given maximal resolution level, the optimal maximum error, i.e., the maximum error when using the global tensor mesh, is preserved by the \NtwoStwo{} LR-mesh. However, the number of B-splines is significantly different and the discrepancy exponentially grows with the maximal resolution level.

%------------------------------------------------------------------------

\section{Application II: Isogeometric analysis}\label{sec:isogeometric-analysis}
Isogeometric analysis (IgA), introduced in \cite{IgA}, is a technique to perform numerical simulations on complex geometries. The numerical solution is represented by means of the same functions used for the domain modeling. Nowadays, complex geometries are expressed in terms of computed aided design (CAD) technologies, such as B-splines, non-uniform rational B-splines (NURBS) and their generalizations to address adaptive refinements.

In this section, we adopt the IgA approach, using our LR refinement strategy, to approximate the solution of the Poisson problem on $\Omega = [0,1]^2$,
\begin{equation}\label{PP}
\left\{\begin{array}{rl}
-\Delta u = f, & \mbox{in } \mathring{\Omega}, \\
u = u_D, & \mbox{on }\partial \Omega,
\end{array}\right.
\end{equation}
whose exact solution is
$$
u(x,y) = \arctan\left(100\left(\sqrt{(x-1.25)^2 + (y+0.25)^2}-\frac{\pi}{3}\right)\right);
$$
see Figure~\ref{PoissonSol}.
This example is a good benchmark for numerical schemes, as the sharp interior layer of the exact solution highlights the approximation quality, and it has been used extensively in the literature; see, e.g., \cite{pozzo2,johannessen,pozzo1}.
\begin{figure}[t!]
\centering
\includegraphics[width=8cm, trim={1cm 4cm 0cm 2.5cm},clip]{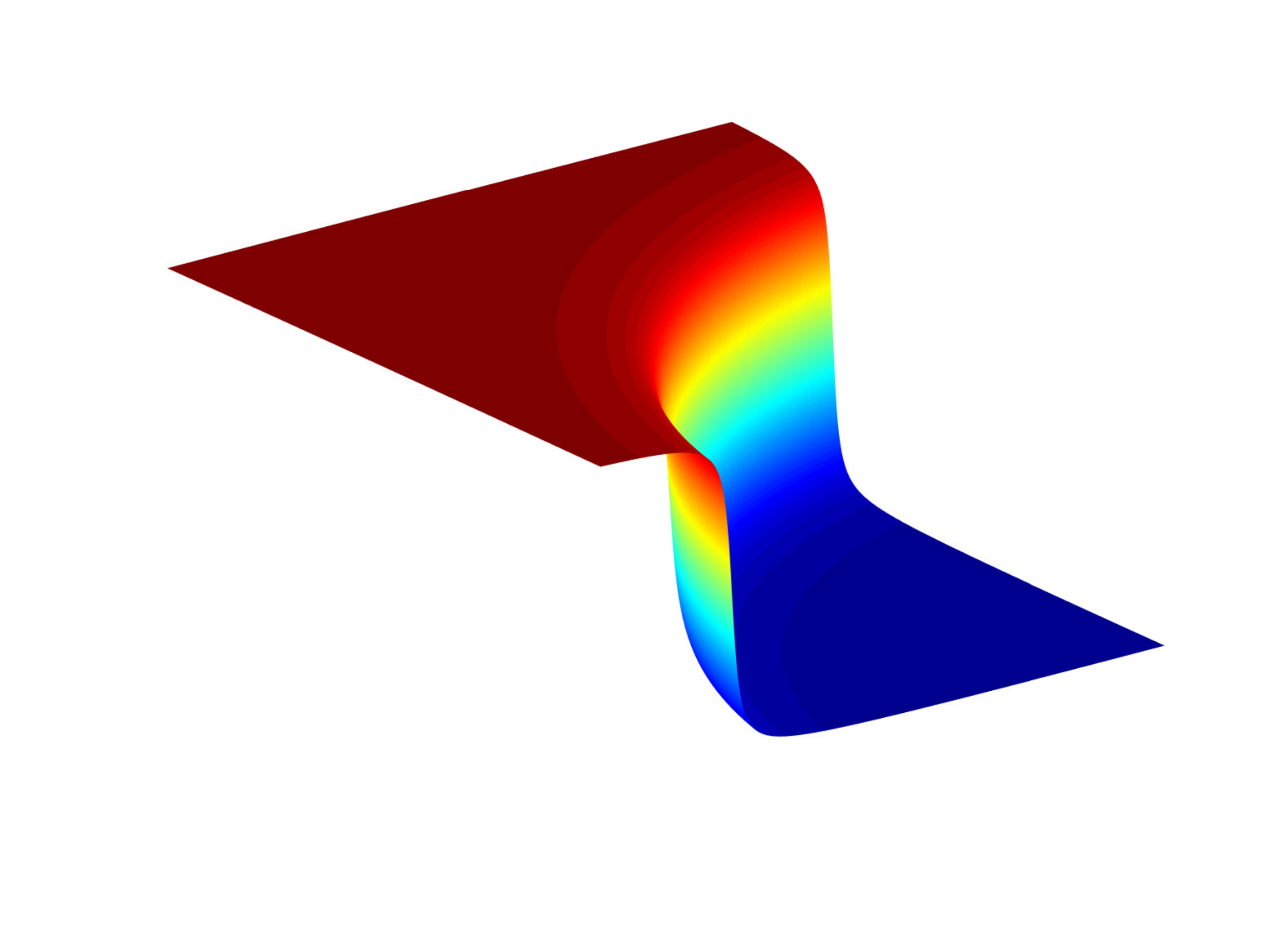}
\caption{Exact solution of the Poisson problem \eqref{PP}.}\label{PoissonSol}
\end{figure}

In the context of Galerkin discretizations, two properties are desirable: \begin{itemize}
\item (local) linear independence of the space generators,
\item refinement adaptivity.
\end{itemize}
The linear independence of the functions used as building blocks of the numerical solution avoids the numerical complexity posed by the singularity of the matrices associated to the problem discretization.
The refinement adaptivity is desired for balancing accuracy and computational cost as it allows for a higher precision, only there where it is needed to reproduce fast variations of the exact solution.
LR B-splines on \NtwoStwo{} LR-meshes are suitable candidates as both the (local) linear independence of the space generators and the adaptivity of the refinement are guaranteed.

\begin{figure}[t!]
\centering
\subfloat[]{
\includegraphics[width=7cm, trim={11.25cm .5cm 11cm 1cm},clip]{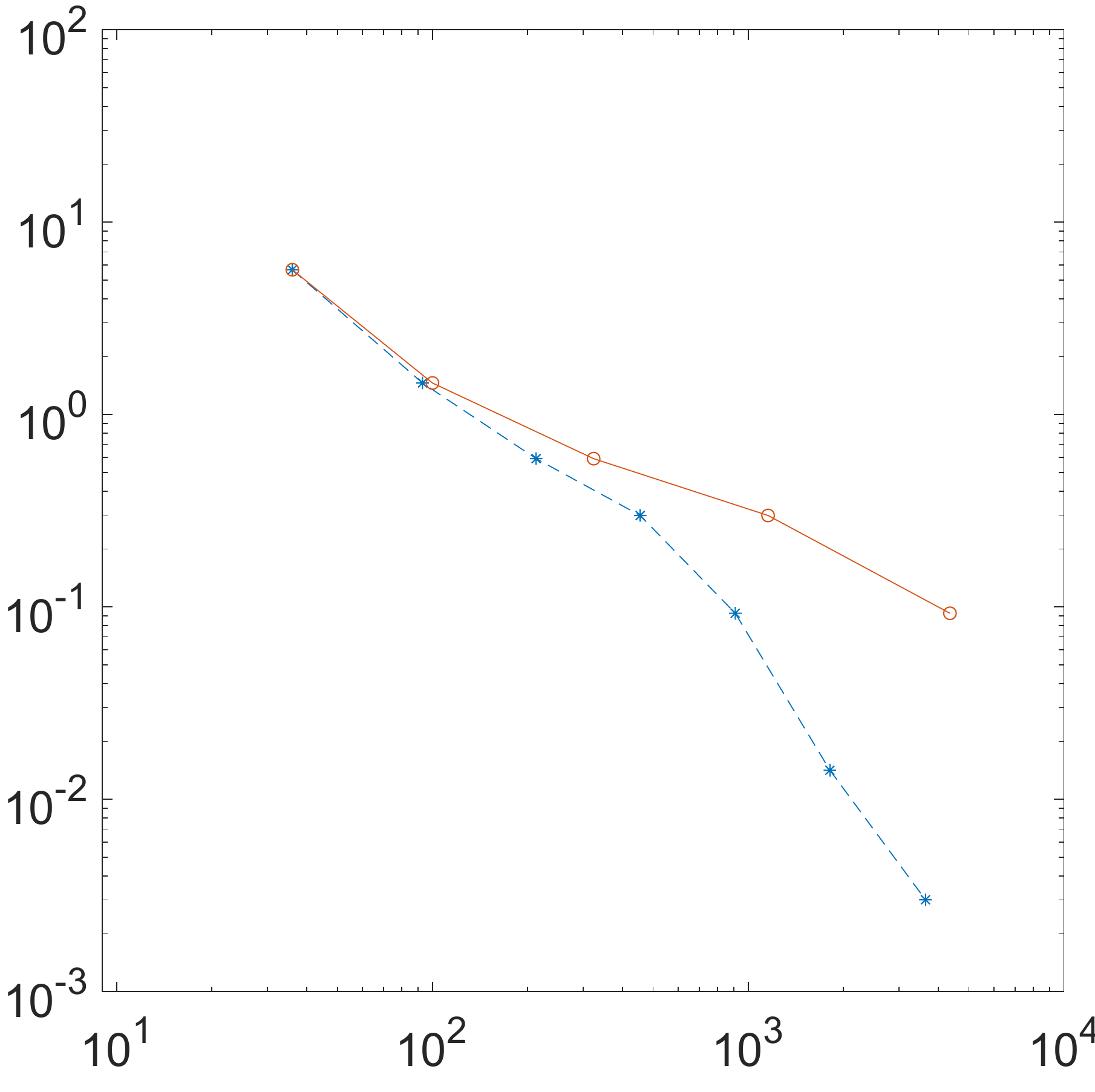}
}\quad
\subfloat[]{
\includegraphics[width=7cm, trim={11.25cm .5cm 11cm 1cm},clip]{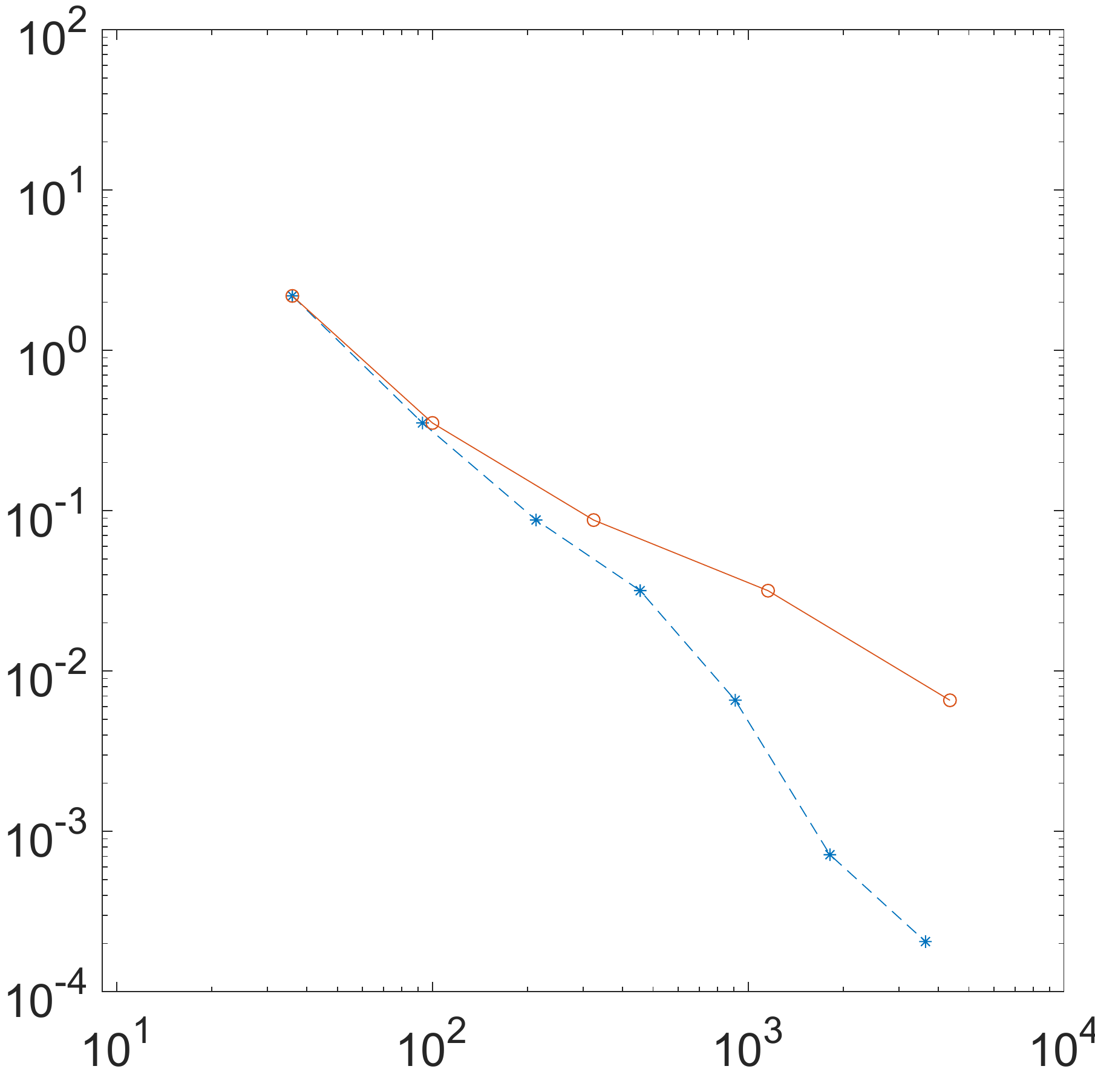}
}
%\subfloat[]{
%\includegraphics[width=5.5cm, trim={12.5cm 2cm 11.5cm 1cm},clip]{PoissonMesh.pdf}
%}
\caption{Decay of the error in (a) the $L^\infty$-norm and (b) the $L^2$-norm, when approximating the solution of problem \eqref{PP} with B-splines of bidegree $(2,2)$ on tensor meshes (solid line) and \NtwoStwo{} LR-meshes (dashed line) for different levels of maximal resolution.
}\label{Poissonerr}
\end{figure}
In Figure~\ref{Poissonerr} we compare the $L^\infty$-norm and the $L^2$-norm of the error (Figures~\ref{Poissonerr}(a) and~\ref{Poissonerr}(b) respectively), using bidegree $(2,2)$ with global tensor meshes and local \NtwoStwo{} LR-meshes for different levels of maximal resolution to approximate the solution of the Poisson problem \eqref{PP}. The \NtwoStwo{} LR-mesh is computed by first applying the structured mesh to the LR B-splines whose supports intersect the curve where the sharp interior layer in the exact solution occurs, and then by performing \unilateral{} tensor expansions to recover the \NtwoS-property.
The error norms, which are computed discretely on a uniform grid of $1000\times 1000$ points, are plotted in log-log scale with respect to the number of LR B-splines on the mesh. The solid line with circular markers shows the decay when using global tensor meshes, whereas the dashed line with star markers shows the decay for the \NtwoStwo{} LR-meshes. In the figures, the first marker corresponds to the $4\times4$ tensor mesh, for maximal resolution level $\ell = 2$, and it is the maximal level for which the LR B-spline and standard tensor B-spline sets coincide. When considering a comparable number of functions, the \NtwoStwo{} LR-mesh leads to a significant reduction of both the $L^\infty$-norm and the $L^2$-norm of the error with respect to the tensor mesh, thanks to the adaptivity of the refinement.

%------------------------------------------------------------------------

\section{Conclusion}\label{sec:conclusion}
LR B-splines are one of the most elegant extensions of univariate B-splines on local tensor structures that allow local refinement.
They possess almost all the properties of classical B-splines, but they are not always linearly independent.
Recently, a characterization of LR-meshes ensuring local linear independence of the corresponding LR B-splines has been presented in the literature. However, a practical adaptive refinement strategy for LR-meshes that maintain such a property was missing.
In this paper, we have filled this gap by describing an adaptive refinement strategy that generates LR-meshes where the corresponding LR B-splines are locally linearly independent.
Subsequently, we have exploited the local linear independence of the LR B-splines to construct efficient quasi-interpolation schemes and to solve elliptic problems using the isogeometric Galerkin method.

%------------------------------------------------------------------------

\section*{Acknowledgements}
We are very grateful to Tor Dokken (SINTEF, Oslo) and Kjetil Andr\'e Johannessen (SINTEF, Trondheim) for sharing with us the example in Figure~\ref{LDstructured} that shows that the standard structured mesh refinement may produce linearly dependent sets of LR B-splines.

%------------------------------------------------------------------------

\begin{appendices}
\section{Nested LR B-splines}\label{sec:appendix}
The purpose of this appendix is to show the equivalence of the definition of nestedness used in this paper (Definition~\ref{nesteddef}) and the original definition provided in \cite[Definition~2.4]{someproperties} for LR B-splines.
The latter definition is formulated in terms of repeated knot insertion and, in view of \cite[Proposition~2.5]{someproperties}, it is equivalent to the  following definition.
\begin{definition}\label{nestedBressan}
Let $B[\pmb{x}^1, \pmb{y}^1]$ and $B[\pmb{x}^2, \pmb{y}^2]$ be two different tensor
B-splines.
Let $\mu_{\pmb{x}^k}(z)$ and $\mu_{\pmb{y}^k}(z)$, for $k=1,2$, denote the number of times $z \in \RR$ occurs in the vectors $\pmb{x}^k$ and $\pmb{y}^k$, respectively. Then, we say that $B[\pmb{x}^2, \pmb{y}^2]$ is \highlight{nested} in $B[\pmb{x}^1, \pmb{y}^1]$, and we write $B[\pmb{x}^2, \pmb{y}^2] \preceq B[\pmb{x}^1, \pmb{y}^1]$, if
\begin{enumerate}
\item\label{nestedBressan:a} $\left\{\begin{array}{ll}
    \mu_{\pmb{x}^2}(z) \geq \mu_{\pmb{x}^1}(z), & \forall\, z \in\; ]x_1^2, x_{p_1+2}^2[, \\[0.2cm]
    \mu_{\pmb{y}^2}(z) \geq \mu_{\pmb{y}^1}(z), & \forall\, z \in\; ]y_1^2, y_{p_2+2}^2[, \\[0.1cm]
    \end{array}
    \right.$
\item\label{nestedBressan:b} $\left\{\begin{array}{ll}
    \mu_{\pmb{x}^2}(z) \leq \mu_{\pmb{x}^1}(z), & \forall\, z \notin\; ]x_1^1,x_{p_1+2}^1[, \\[0.2cm]
    \mu_{\pmb{y}^2}(z) \leq \mu_{\pmb{y}^1}(z), & \forall\, z \notin\; ]y_1^1, y_{p_2+2}^1[. \\[0.1cm]
    \end{array}
    \right.$
\end{enumerate}
\end{definition}

We now prove the equivalence of the two definitions for LR B-splines.
\begin{prop}
 Given a mesh $\cM$, let $B[\pmb{x}^1, \pmb{y}^1]$ and $B[\pmb{x}^2, \pmb{y}^2]$ be two different LR B-splines defined on $\cM$. For $B[\pmb{x}^1, \pmb{y}^1]$ and $B[\pmb{x}^2, \pmb{y}^2]$, Definition~\ref{nesteddef} is equivalent to Definition~\ref{nestedBressan}.
\end{prop}
\begin{proof}
Let $B^1:=B[\pmb{x}^1, \pmb{y}^1]$ and $B^2:=B[\pmb{x}^2, \pmb{y}^2]$ be two LR B-splines defined on the mesh $\cM$. Assuming that they satisfy the conditions in Definition~\ref{nestedBressan}, we prove that they also satisfy the conditions in Definition~\ref{nesteddef}. Let us first show that $\supp B^2 \subseteq \supp B^1$. This means that $[x_1^2, x_{p_1+2}^2]\subseteq [x_1^1, x_{p_1+2}^1]$ and $[y_1^2, y_{p_2+2}^2]\subseteq [y_1^1, y_{p_2+2}^1]$. Suppose $x_1^2 < x_1^1$. Then, $\mu_{\pmb{x}^1}(x_1^2) = 0$ and $\mu_{\pmb{x}^2}(x_1^2) > 0$, but this contradicts item~\ref{nestedBressan:b} in Definition \ref{nestedBressan}, and hence $x_1^2 \geq x_1^1$. The other inequalities can be proved in a similar way to have the interval inclusions.
Let now $\gamma \subseteq \partial \supp B^1 \cap \partial \supp B^2$. Assume without loss of generality that $\gamma$ is a $1$-meshline. Then, $\gamma$ is a $(1,z)$-meshlines for some $z \in \{x_1^1, x_{p_1+2}^1\}$. For any choice of such $z$, we have $\mu_{\pmb{x}^2}(z) \leq \mu_{\pmb{x}^1}(z)$, by item~\ref{nestedBressan:b} of Definition \ref{nestedBressan}, and, by the definition of multiplicity of $\gamma$ in $\cM(\pmb{x}^1, \pmb{y}^1)$ and $\cM(\pmb{x}^2, \pmb{y}^2)$, this implies that $\mu(\gamma)$ is higher (or equal) in $\cM(\pmb{x}^1, \pmb{y}^1)$ than in $\cM(\pmb{x}^2, \pmb{y}^2)$. Therefore, both the conditions in Definition~\ref{nesteddef} are satisfied if those in Definition~\ref{nestedBressan} are satisfied.

Let us now show the converse. Assuming that the conditions in Definition~\ref{nesteddef} are fulfilled, we prove that the conditions in Definition~\ref{nestedBressan} are satisfied as well. 
Let $z \in\; ]x_1^2, x_{p_1+2}^2[$. 
Since $\supp B^2 \subseteq \supp B^1$, we have $]x_1^2, x_{p_1+2}^2[ \;\subseteq\; ]x_1^1, x_{p_1+2}^1[$. If $z \notin \pmb{x}^1$, then $\mu_{\pmb{x}^1}(z) = 0$ and therefore $\mu_{\pmb{x}^2}(z) \geq \mu_{\pmb{x}^1}(z)$ for any value of $\mu_{\pmb{x}^2}(z)$. If $z\in\pmb{x}^1$, then it must be also in $\pmb{x}^2$, otherwise the $(1,z)$-split $\{z\}\times [y_1^1, y_{p_2+2}^1]$ would traverse $B^2$, which would not have minimal support. For the same reason, it must also hold that $\mu_{\pmb{x}^2}(z) = \mu_{\pmb{x}^1}(z)$. This proves item~\ref{nestedBressan:a} of Definition \ref{nestedBressan}.
Assume now $z \notin [x_1^1,x_{p_1+2}^1]$. Since $\supp B^2 \subseteq \supp B^1$, we have $x_1^1 \leq x_1^2$ and $x_{p_1+2}^2 \leq x_{p_1+2}^1$. Therefore, $\mu_{\pmb{x}^1}(z) = \mu_{\pmb{x}^2}(z) = 0$. If $z \in \{x_1^1, x_{p_1+2}^1\}$ but $z \notin \pmb{x}^2$, then trivially $\mu_{\pmb{x}^2}(z) \leq \mu_{\pmb{x}^1}(z)$. If $z \in \pmb{x}^2$, then $z$ corresponds to $(1,z)$-meshlines in $\partial \supp B^1 \cap \partial \supp B^2$. By assumption, these meshlines have a higher (or equal) multiplicity in $\cM(\pmb{x}^1, \pmb{y}^1)$ than in $\cM(\pmb{x}^2, \pmb{y}^2)$, which means $\mu_{\pmb{x}^2}(z) \leq \mu_{\pmb{x}^1}(z)$. The same line of arguments also applies to the knots in the second direction. This proves item~\ref{nestedBressan:b} of Definition~\ref{nestedBressan}.
\end{proof}

\end{appendices}

%------------------------------------------------------------------------

\bibliography{biblio}
\end{document}